\documentclass[11pt]{amsart}
\usepackage{fullpage}

\usepackage{amsthm}
\usepackage{amssymb}
\usepackage{ytableau}
\usepackage{tikz-cd}
\usepackage{mathtools}
 \newtheorem{thm}{Theorem}[section]
 \newtheorem{cor}[thm]{Corollary}
 \newtheorem{lem}[thm]{Lemma}
 \newtheorem{prop}[thm]{Proposition}
 \theoremstyle{definition}
 \newtheorem{defn}[thm]{Definition}
 \newtheorem{rem}[thm]{Remark}
 
\newtheorem{ex}[thm]{Example}

 \numberwithin{equation}{section}
\DeclareMathOperator{\Ima}{Im}
\DeclareMathOperator{\Hom}{Hom}
\DeclareMathOperator{\coker}{Coker}
\DeclareMathOperator{\sign}{sgn}
\DeclareMathOperator{\spn}{span}
\begin{document}

\title{On the free LAnKe on $3n-2$ generators: a theorem of Friedmann, Hanlon, Stanley and Wachs}

%----------Author 1

\author[Maliakas]{Mihalis Maliakas}

\address{%
	Department of Mathematics\\
	University of Athens\\
	Greece}
\email{mmaliak@math.uoa.gr}

\author[Stergiopoulou]{Dimitra-Dionysia Stergiopoulou}
\address{Department of Mathematics\\
	University of Athens\\
	Greece}
\email{dstergiop@math.uoa.gr}



\date{December 08, 2025}

%%% ----------------------------------------------------------------------

\begin{abstract}A LAnKe (also known as a Filippov algebra or a Lie algebra of the $n$-th kind) is a vector space equipped with a skew-symmetric $n$-linear form that satisfies the generalized Jacobi identity. Friedmann, Hanlon, Stanley and Wachs have shown that the symmetric group acts on the multilinear part of the free LAnKe on $2n-1$ generators as an irreducible representation. They announced that the multilinear component on $3n-2$ generators decomposes as a direct sum of two irreducible symmetric group representations and a proof was given recently in a subsequent paper by Friedmann, Hanlon and Wachs. In the present paper we provide a proof of the later statement. The two proofs are substantially different.
\end{abstract}

%%% ----------------------------------------------------------------------
\maketitle
%%% ----------------------------------------------------------------------
%\tableofcontents

\section{Introduction} Throughout this paper we work over a field $\mathbb{K}$ of characteristic zero. 

Since the mid 1980's various $n$-ary generalizations of Lie algebras have been introduced and studied. We refer to the Introduction of the paper by Friedmann, Hanlon, Stanley and Wachs \cite{FHSW} for a discussion of such generalizations, historical background and ties with various areas of mathematics and physics. Also there is an extensive review by de Azcárraga and  Izquierdo \cite{AI}. Among the above generalizations are the Filippov algebras \cite{Fi}, which are also called LAnKes \cite{FHSW}.
\begin{defn}[{\cite[Definition 1.2]{FHSW}}]\label{lanke} A \textit{Lie algebra of the $n$-th kind} (\textit{LAnKe or Filippov algebra}) is a $\mathbb{K}$-vector space $\mathcal{L}$ equipped with an $n$-linear bracket $ [-,-, \dots, -] :  \mathcal{L}^n \to \mathcal{L}$ such that for all $x_1,\dots, x_n, y_1, \dots, y_{n-1} \in \mathcal{L}$,
	\begin{enumerate}
		\item $[x_1, x_2, \dots, x_n]=\sign(\sigma)[x_{\sigma(1)}, x_{\sigma(2)}, \dots, x_{\sigma(n)}$] for every $\sigma \in \mathfrak{S}_n$, and
		\item the following generalized Jacobi identity holds\begin{align}\label{GJI}
			&[[x_1,x_2,\dots,x_n],y_1,\dots, y_{n-1}]\\\nonumber&
			=\sum_{i=1}^{n}[x_1,x_2,\dots,x_{i-1},[x_i,y_1,\dots, y_{n-1}],x_{i+1},\dots, x_n].
		\end{align}
	\end{enumerate}
\end{defn}

Homomorphisms between LAnKes are defined in the usual way. Following \cite[Definition 2.1]{FHSW}, the
\textit{free} LAnKe on a set $X$ is a LAnKe $\mathcal{L}$ together with a map $i:X \to \mathcal{L}$ such that if $f:X \to \mathcal{L'}$ is a map, where $\mathcal{L'}$ is a LAnKe, then there is a unique LAnKe homomorphism $F: \mathcal{L} \to \mathcal{L'}$ such that $f=F \circ i$. By a standard argument, free LAnKes on $X$ are isomorphic. It is clear that the free LAnKe for $n=2$ is the free Lie algebra.

In \cite{FHSW}, Friedmann, Hanlon, Stanley and Wachs initiated the study of the action of the symmetric group $\mathfrak{S}_m$ on the multilinear component of the free LAnKe. To be precise, the \textit{multilinear component} $\text{Lie}_n(m)$ of the free LAnKe on $[m]:=\{1, \dots, m\}$ is spanned by the bracketed words on $[m]$ in which each $i$ appears exactly once. It follows that each such bracketed word has the same number of brackets, say $k$, and $m=(n-1)k+1$. Consider the action of $\mathfrak{S}_{m}$ on $\text{Lie}_n(m)$ given by replacing $i$ by $\sigma(i)$ in each bracketed word. Let us denote the corresponding representation of $\mathfrak{S}_m$ by $\rho_{n,k}$. 

For a partition $\lambda$ of $m$, let $S^\lambda$ be the corresponding Specht module of the symmetric group $\mathfrak{S}_m$. As $\lambda$ ranges over all partitions of $m$, the modules $S^\lambda$ form a complete set of inequivalent representations of $\mathfrak{S}_m$. It was shown by Friedmann, Hanlon, Stanley and Wachs in \cite[Theorem 1.3]{FHSW} that the representation $\rho_{n,2}$  of $\mathfrak{S}_{2n-1}$ is isomorphic to the Specht module $S^{(2^{n-1},1)}$ if $n \ge 2$.  The following result was announced in \cite{FHSW} and \cite{FHSW0}. A proof appeared in the recent paper \cite{FHW1}. 
\begin{thm}[{\cite[Theorem 3]{FHSW0}, \cite[p.4]{FHSW}, \cite[Theorem 1.3]{FHW1}}]\label{mainspecht} The representation $\rho_{n,3}$ of $\mathfrak{S}_{3n-2}$ is isomorphic to the direct sum $S^{(3^{n-2},2,1^2)} \oplus S^{(3^{n-1},1)}$ for every $n \ge 2$.
\end{thm}
The decomposition of $\rho_{n,k}$ into irreducibles has been obtained for $k=4$ recently in \cite{FHW1} and \cite{MS7} (the two proofs are substantially different). The decomposition of $\rho_{n,k}$ remains open for $k \ge 5$.

The purpose of this paper                                                                                                                                                                                                                                                                                                                                                                                                                            is to  prove Theorem \ref{mainspecht}. We approach the problem within the framework of representations of the general linear group $G=GL_N(\mathbb{K})$. Using ideas from \cite{MMS}, which in turn were based on ideas of Brauner, Friedmann, Hanlon, Stanley and Wachs in \cite{{BF},{FHSW},{FHW}}, we define and study a particular $G$-equivariant map
\begin{equation}\label{intro2}
	\Lambda(\lambda) \oplus \Lambda(\lambda) \oplus  \Lambda(\nu) \xrightarrow{\Omega(\gamma_1) + \Omega(\gamma_2)+ \Omega(\gamma_3)} \Lambda(\lambda).
\end{equation}
where $\Lambda(\lambda)= \Lambda^nV \otimes \Lambda^{n-1}V \otimes \Lambda^{n-1}V$ and $\Lambda(\nu)= \Lambda^nV \otimes \Lambda^{n}V \otimes \Lambda^{n-2}V$  are the tensor products of the indicated exterior powers of the natural $G$-module of column vectors and $N \ge 3n-2$. This map has the property that when we apply the Schur functor, we obtain a presentation of the multilinear component $\text{Lie}_n(m)$ of the free LAnKe, where $m=3n-2$. We analyze the effect of the map on irreducible summands of $\Lambda(\lambda)$ using combinatorics of semistandard tableaux.

In Section 2 we establish notation and gather some recollections to be used in the sequel. In Section 3 we analyze the map \[\Lambda(\lambda) \oplus \Lambda(\lambda) \xrightarrow{\Omega(\gamma_1) + \Omega(\gamma_2)} \Lambda(\lambda)\] and determine the irreducible decomposition of its cokernel (see Theorem \ref{main1}). In Section 4 we extend the analysis to the map (\ref{intro2}). The main result of this paper for $G$ describes the irreducible decomposition of the cokernel of (\ref{intro2}) (see Theorem \ref{main2}). In Section 5 we determine a presentation of $\text{Lie}_n(m)$ (see Lemma \ref{presentationlanke}). Using this and the Schur functor we show that Theorem \ref{mainspecht} follows from Theorem \ref{main2}.
\section{Preliminaries}
The purpose of this section is to establish notation and discuss results that will be used in the sequel. Our main references here are the books by Fulton \cite{F} and  Weyman \cite{W} and the paper \cite{AB} by Akin and Buchsbaum.
\subsection{Divided power algebra and exterior algebra}
Let $G=GL_N(\mathbb{K})$. We denote by $V=\mathbb{K}^N$  the natural $G$-module consisting of column vectors.

By $D=\bigoplus_{i\geq 0}D_i$ we denote the divided power algebra of $V$. We will recall some definitions and facts concerning this algebra. For more details we refer to \cite[Section 1.1]{W}. 

We recall that $D$ is defined as the graded dual of the symmetric algebra $S(V^*)$ of $V^*$, where $V^*$ is the dual of $V$. So by definition we have \[D_i = (S_i(V^*))^*.\] 

Since the characteristic of $\mathbb{K}$ is zero, $D$ is naturally isomorphic to the symmetric algebra $SV$ of $V$. However, the computations to be made in Sections 3 and 4.1 seem less involved if one deals with Weyl modules in place of Schur modules. For this reason we work with the divided power algebra and Weyl modules.

	If $v \in V$ and $i$ is a nonnegative integer, we have the $i$th divided power  $v^{(i)} \in D_i$ of $v$. In particular, \[ v^{(0)}=1 \ \mathrm{and} \ v^{(1)}=v\] for all $v \in V$. We recall that if $i,j$ are nonnegative integers, then the product $v^{(i)}v^{(j)}$ of $v^{(i)}$ and $v^{(j)}$ in $D$ is given by \[v^{(i)}v^{(j)}=\tbinom{i+j}{j}v^{(i+j)},\] where $\tbinom{i+j}{j}$ is the indicated binomial coefficient. These relations will be used many times in Sections 3 and 4.1.

If $\{e_1, \dots, e_N\}$ is a basis of the vector space $V$, then a basis of the vector space $D_i$ is the set \[\{e_1^{(\alpha_1)}\cdots e_N^{(\alpha_N)}: \alpha_1+\cdots + \alpha_N = i\}.\]

We recall that $D$ has a graded Hopf algebra structure. Let \[\Delta : D \to D \otimes D\] be the comultiplication map of $D$. Explicitly, for a homogeneous element $x= v_1^{(\alpha_1)}\cdots v_t^{(\alpha_t)} \in D_a$, where $v_i \in V$, we have \[ \Delta(x)=\sum_{0\le \beta_i \le \alpha_i} v_1^{(\beta_1)}\cdots v_t^{(\beta_t)} \otimes v_1^{(\alpha_1 - \beta_1)}\cdots v_t^{(\alpha_t - \beta_t)}.\]
For $0 \le b \le a$ we may restrict the above sum to those $\beta_i$ such that $\beta_1 + \cdots +  \beta_t= b$. This yields the following component of the comultiplication map \begin{align*} D_a &\to D_b \otimes D_{a-b}, \\  x &\mapsto \sum_{\substack{0\le \beta_i \le \alpha_i \\ \beta_1+\cdots+\beta_t=b}} v_1^{(\beta_1)}\cdots v_t^{(\beta_t)} \otimes v_1^{(\alpha_1 - \beta_1)}\cdots v_t^{(\alpha_t - \beta_t)}, \end{align*}
which we will again denote simply by $\Delta:  D_a \to D_b \otimes D_{a-b}$ in order to avoid cumbersome notation.

By coassociativity of the comultiplication map  $\Delta : D \to D \otimes D$, the compositions \begin{align*}D \xrightarrow{\Delta} D \otimes D \xrightarrow{1\otimes \Delta}D \otimes (D \otimes D),\\
D \xrightarrow{\Delta} D \otimes D \xrightarrow{\Delta \otimes 1}(D \otimes D) \otimes D	
 \end{align*}are equal. We refer to this map as the twofold comultiplication map $D \to D \otimes D \otimes D$. The component $D_a \to D_{a_1} \otimes D_{a_2} \otimes D_{a_3}$ of this map, where $a_i$ are nonnegative integers such that $a=a_1+a_2+a_3$, is given as follows,
\begin{align*} D_a &\to D_{a_1} \otimes D_{a_2} \otimes D_{a_3}, \\  x &\mapsto \sum_{\substack{0\le \beta_i + \gamma_i \le \alpha_i \\ \beta_1+\cdots+\beta_t=a_2\\\gamma_1 + \cdots \gamma_t = a_1}} v_1^{(\gamma_1)}\cdots v_t^{(\gamma_t)} \otimes v_1^{(\beta_1)}\cdots v_t^{(\beta_t)} \otimes v_1^{(\alpha_1 - \beta_1 - \gamma_1)}\cdots v_t^{(\alpha_t - \beta_t - \gamma_t)}, \end{align*}
where $x= v_1^{(\alpha_1)}\cdots v_t^{(\alpha_t)} \in D_a$.

By $\Lambda = \bigoplus_{i \ge 0} \Lambda^i$ we denote the exterior algebra of $V$. We recall that $\Lambda$ has a graded Hopf algebra structure. If $u,v \in \Lambda$, we denote their product in  $\Lambda$ by $uv$. If $\{e_1, \dots, e_N\}$ is a basis of the vector space $V$, then a basis of the vector space $\Lambda^i$ is the set \[\{e_{\alpha_1} e_{\alpha_{2}}\cdots  e_{\alpha_i}: 1 \le \alpha_1 < \cdots < \alpha_i \le N\}.\] We denote by \[\Delta : \Lambda \to \Lambda \otimes \Lambda\] the comultiplication map of $\Lambda$. Explicitly, for a homogeneous element $x= v_{1} v_{2}\cdots  v_{a} \in \Lambda^a$, where $v_i \in V$, we have \[ \Delta(x)=\sum_{0 \le s \le a}\sum_{\sigma} \sign(\sigma)v_{\sigma(1)} \dots v_{\sigma(s)} \otimes v_{\sigma(s+1)} \dots v_{\sigma(a)},\] where the second sum is over all permutations $\sigma$ of $\{1, \dots, a\}$ such that $\sigma(1) < \dots <\sigma(s)$ and $\sigma(s+1) < \dots < \sigma(a)$.

For $0 \le b \le a$ we have the following component of the comultiplication map \begin{align*} \Lambda^a &\to \Lambda^b \otimes \Lambda^{a-b}, \\  x &\mapsto \sum_{\sigma} \sign(\sigma)v_{\sigma(1)} \dots v_{\sigma(b)} \otimes v_{\sigma(b+1)} \dots v_{\sigma(a)}, \end{align*} where the sum is over all permutations $\sigma$ of $\{1, \dots, a\}$ such that $\sigma(1) < \dots <\sigma(b)$ and $\sigma(b+1) < \dots < \sigma(a)$. We will  denote this map simply by $\Delta:  \Lambda^a \to \Lambda^b \otimes \Lambda^{a-b}$.

By associativity of the multiplication map $\Lambda \otimes \Lambda \to \Lambda$, we have a well defined map $\Lambda \otimes \Lambda \otimes \Lambda \to \Lambda$ which we refer to as the twofold multiplication map.

We have used the same symbol $\Delta$ for the comultiplication maps in the algebra $D$ and $\Lambda$. In the sequel it will be clear which algebra is considered each time. If there is a need of distinction, we will write $\Delta_D$ and $\Delta_\Lambda$.

\subsection{Partitions, Weyl modules and Schur modules}For a positive integer $r$, let $\Lambda(N,r)$ be the set of sequences $\alpha=(\alpha_1, \dots, \alpha_N)$ of length $N$ of nonnegative integers such that $\alpha_1+\cdots +\alpha_N= r$ and let  $\Lambda^+(N,r)$ be the subset of $\Lambda(N,r)$ consisting of partitions, that is sequences $\mu=(\mu_1, \dots, \mu_N)$ such that $\mu_1 \ge \mu_2 \ge \dots \ge \mu_N$. The length $\ell(\mu)$ of a partition $\mu=(\mu_1, \dots, \mu_N)$ is the maximum $s$ such that $\mu_s \neq 0$.

For $\alpha=(\alpha_1,\dots, \alpha_N) \in \Lambda(N,r)$, let \[D(\alpha):=D(\alpha_1,\dots,\alpha_N)\] be the tensor product $D_{\alpha_1}\otimes \dots \otimes D_{\alpha_N}$ over $\mathbb{K}$. Likewise, for exterior powers let \[\Lambda(\alpha):=\Lambda(\alpha_1,\dots,\alpha_N)\] be the tensor product $\Lambda^{\alpha_1}\otimes \dots \otimes \Lambda^{\alpha_N}$ over $\mathbb{K}$.

For $\mu \in \Lambda^+(N,r)$, we denote by $K_\mu$ the corresponding Weyl module for $G$ and by $L_\mu$ the corresponding Schur module for $G$ defined in \cite[Section 2.1]{W}.  For example, when $\mu=(a)$ consists of one part, then $K_{(a)}=D_a$ and $L_{(a)}=\Lambda^{a}$. If $\mu=(1^a)$, then $K_{(a)}=\Lambda^a$ and $L_{(a)}=S_{a}$, where the later module is the degree $a$ symmetric power of the natural module $V$.

Since the characteristic of $\mathbb{K}$ is zero, for every $\mu \in \Lambda^+(N,r)$ the $G$-modules $K_\mu$ and $L_{\mu'}$ are isomorphic irreducible modules, where $\mu'$ denotes the conjugate partition of $\mu$, see \cite[Section 2.2]{W}.

It is a classical fact that the multiplicity of $K_\mu$ in $D(\mu)$ is equal to 1, see \cite[Corollary 2(a), Section 8.3]{F}. We denote by \[{\pi}_{\mu} : D(\mu) \to K_\mu\] the natural projection (which is unique up to a nonzero scalar multiple).

\subsection{Tableaux and semistandard basis}\label{2.2} Let us recall an  important combinatorial property of $K_{\mu}$. 

We fix the order $e_1<e_2< \cdots <e_N$ on the natural basis $\{e_1,\dots,e_N\}$ of $V$. In the sequel we will denote  $e_i$ by its subscript $i$. If $\mu=(\mu_1,\dots,\mu_N) \in \Lambda^+(N,r)$, a \textit{tableau} of shape $\mu$ is a filling of the diagram of $\mu$ with entries from $\{1,\dots,N\}$. A tableau is called \textit{row semistandard} if the entries are weakly increasing across the rows from left to right.  A row semistandard tableau is called \textit{semistandard} if the entries are strictly increasing down each column. We denote the set of row semistandard tableaux (respectively, semistandard  tableaux) of shape $\mu$ by $\mathrm{RSST}(\mu)$ (respectively, $\mathrm{SST}(\mu)$). The \textit{weight} of a tableau $S$ is the tuple $\alpha=(\alpha_1,\dots,\alpha_N)$, where $\alpha_i$ is the number of appearances of the entry $i$ in $S$. The set consisting of the semistandard (respectively, row semistandard) tableaux of shape $\mu$ and weight $\alpha$ will be denoted by $\mathrm{SST}_{\alpha}(\mu)$ (respectively, $\mathrm{RSST}_{\alpha}(\mu)$). For example, the following tableau of shape $\mu=(5,4,2)$
\[ \ytableausetup{smalltableaux}\begin{ytableau}
	\ 1&1&1&1&2\\
	\ 2&2&4&4 \\
	\ 3&4
\end{ytableau} \]
is semistandard and has weight $\alpha=(4,3,1,3)$. We will use `exponential' notation for row semistandard tableaux.

If $\mu = (\mu_1, \mu_2, \dots, \mu_N)$ is a partition and $S$ is a row semistandard tableau of shape $\mu$,
\[ S=\begin{matrix*}[l]
	1^{(a_{11})} \cdots N^{(a_{1N})} \\
	1^{(a_{21})} \cdots N^{(a_{2N})} \\
	\ \ \vdots \ \ \ \ \ \ \ \ \vdots \\
	1^{(a_{N1})} \cdots N^{(a_{NN})} \end{matrix*},\]
where $a_{ij}$ are nonnegative integers, let $e^S \in D(\mu)$ be the element \[e^S:=1^{(a_{11})} \cdots N^{(a_{1N})}  \otimes 1^{(a_{21})} \cdots N^{(a_{2N})} \otimes 1^{(a_{N1})} \cdots N^{(a_{NN})}\] obtained by `reading the rows' of $S$ from left to right and top to bottom. We note that \[\sum_ja_{ij} = \mu_i \ (i=1, \dots, N) \ \ \mathrm{and} \ \ \sum_ia_{ij} = \ \alpha_j \ (j=1, \dots N), \]
where $\alpha = (\alpha_1, \dots, \alpha_N)$ is the weight of $S$.

A classical result here is the following, see \cite[(2.1.15) Proposition]{W}. \begin{thm}\label{sbthm} Let $\mu \in \Lambda^+(N,r)$. Then there is a bijection between $\mathrm{SST}(\mu)$, and a basis of the $\mathbb{K}$-vector space $K_{\mu}$ given by 
$ S \mapsto \pi_{\mu}(e^S).$\end{thm}
We refer to the elements of this basis of $K_\mu$ as \textit{semistandard basis elements}.
\subsection{Straightening row semistandard tableaux} Let $\mu$ be a partition and $S \in \mathrm{RSST}(\mu)$. For $j \in \{1,2,\dots, \ell(\mu)-1\}$, consider the tableau \[S[j,j+1]\] consisting of rows $j$ and $j+1$ of $S$. We have the partition $(\mu_j, \mu_{j+1})$ consisting of rows $j$ and $j+1$ of $\mu$ and we have the corresponding Weyl module $K_{(\mu_j,\mu_{j+1})}$. From the last paragraph of the proof of \cite[(2.1.15) Proposition]{W} we have the following result.
\begin{lem}\label{insertrows}
	Let $\mu$ be a partition and $j \in \{1,2,\dots, \ell(\mu)-1\}$.  Let $S \in \mathrm{RSST}(\mu)$ be a row semistandard tableaux. If in $K_{(\mu_j,\mu_{j+1})}$ we have a linear combination \[\pi_{(\mu_j,\mu_{j+1})}(e^{S[j,j+1]})=\sum_ic_i\pi_{(\mu_j,\mu_{j+1})}(e^{S[j,j+1]_i}),\] where $c_i \in \mathbb{K}$ and $S[j,j+1]_i$ are row semistandard tableaux of shape $(\mu_j,\mu_{j+1})$, then in $K_\mu$ we have the linear combination \[\pi_{\mu}(e^{S})=\sum_ic_i\pi_{\mu}(e^{S_i}),\] where $S_i$ is the tableau obtained from $S$ by replacing rows $j$ and $j+1$ with $S[j,j+1]_i$.
\end{lem}

Roughly speaking, the previous lemma allows us to obtain relations in $K_\mu$ from relations involving any pair of consecutive rows of a tableau $S \in \mathrm{RSST}(\mu)$. 

In the sequel we will need to express elements of Weyl modules as explicit linear combinations of semistandard basis elements.   To this end, we will apply many times the above lemma together with the next lemma which concerns violations of semistandardness in the first column. 
\begin{lem}[{\cite[Lemma 4.2]{MS3}}]\label{lemglas}Let $\nu=(\nu_1,\nu_2)$ be a partition of length two and let \[S=\begin{matrix*}[l]
		1^{(a_1)}2^{(a_2)}  \cdots   N^{(a_N)} \\
		1^{(b_1)}2^{(b_2)}  \cdots  N^{(b_N)}
	\end{matrix*} \in \mathrm{RSST} (\nu).\]
	Then we have the following identities in $K_{\nu}$.
	\begin{enumerate}
		\item If $a_1+b_1>\nu_1$, then $\pi_\nu(e^S)=0$.
		\item If $a_1+b_1 \le \nu_1$, then 
		\begin{equation}\label{eqglas}\pi_\nu(e^S)=(-1)^{b_1}\sum_{k_2,\dots,k_N}\tbinom{b_2+k_2}{b_2}\cdots\tbinom{b_N+k_N}{b_N}
	\pi_\nu(e^{S(k_2, \dots, k_N)}),
		\end{equation} where \[S(k_2,\dots,k_N)= \begin{matrix*}[l]
			1^{(a_1+b_1)}2^{(a_2-k_2)}  \cdots   N^{(a_N-k_N)} \\
			2^{(b_2+k_2)}  \cdots  N^{(b_N+k_N)}
		\end{matrix*}\] and the sum ranges over all  nonnegative integers $k_2,\dots,k_N$
		such that  $k_2+\dots+k_N=b_1 $ and $k_s \le a_s$ for all $s=2,\dots,N$.	\end{enumerate}	
\end{lem}
 We may think of the sum in the right hand side of eq. (\ref{eqglas}) as been taken over all ways of replacing the $b_1$ 1's in the second row of the tableaux $S$ with $k_2$ 2's, $k_3$ 3's, $\dots$, $k_N$ N's from the first row of $S$, where $k_2+ k_3+ \dots + k_N=b_1$. 
 
 Even though our paper \cite{MS3} concerns modular representations, the proof of the above lemma given there is valid for any field in place of $\mathbb{K}$ (in fact for any commutative ring). In \cite[Lemma 4.2]{MS3} we used the notation $\Delta_\nu$ for the Weyl module $K_\nu$. 

We refer to the first equality of part (2) of Lemma \ref{lemglas} as \textit{raising the 1's from  row 2 of $S$ to row 1}. If the number of rows of the tableau $S \in \mathrm{RSST}(\mu)$ is greater than 2, then according to Lemma \ref{insertrows} we may apply Lemma \ref{lemglas} to any pair $(j,j+1)$ of consecutive rows of $S$ to \textit{raise the 1's from row $j+1$ of $S$ to row $j$}. By repeating this process a finite number of times, we may raise all the 1's to  row 1. (If the total number of 1's in $S$ is strictly greater than the length of the first row of $S$, then $\pi_\mu(e^S)=0$ by the first part of Lemma \ref{lemglas}.)

As an illustration of raising the 1's, we consider the following example. Note that at the end of this example we raise the 2's from row 3 to row 2.

\begin{ex}\label{ex} Let $\mu=(4,3,2)$ and \[ S=\begin{matrix*}[l]
		1 2^{(2)}3 \\
		123 \\
		13 \end{matrix*} \in \mathrm{RSST}(\mu).\] We will express $\pi_\mu(e^S)$ as a linear combination of semistandard basis elements of the Weyl module $K_\mu$ using Lemma \ref{lemglas}.
		
		Let $\nu=(\mu_2,\mu_3)=(3,2)$. Applying Lemma \ref{lemglas}(2) for the tableau \[ S[2,3]=\begin{matrix*}[l] 
			123 \\
			13 \end{matrix*} \in \mathrm{RSST}(\nu)\] we obtain \[ \pi_{\nu}(123 \otimes 13)=-\pi_{\nu}(1^{(2)}3 \otimes 23) - \tbinom{2}{1}\pi_{\nu}(1^{(2)}2 \otimes 3^{(2)}) 
		\]
		and thus from Lemma \ref{insertrows} we get \begin{equation}\label{exeq1} \pi_\mu(e^S)=\pi_{\mu}(1 2^{(2)}3 \otimes 123 \otimes 13)=-\pi_{\mu}(1 2^{(2)}3 \otimes 1^{(2)}3 \otimes 23) - \tbinom{2}{1}\pi_{\mu}(1 2^{(2)}3 \otimes 1^{(2)}2 \otimes 3^{(2)}). 
		\end{equation} Thus we have raised the 1's from  row 3 of $S$ to row 2. 
		
		Next, in each summand of the right hand side of eq. (\ref{exeq1}), we raise the 1's from row 2 to row 1. 
		
		Let $\nu=(\mu_1,\mu_2)=(4,3)$. Consider the first summand of the right hand side of eq. (\ref{exeq1}). Applying Lemma \ref{lemglas}(2) for the tableau \[ \begin{matrix*}[l] 
			12^{(2)}3 \\
			1^{(2)}3 \end{matrix*} \in \mathrm{RSST}(\nu)\] we obtain \[ \pi_{\nu}(12^{(2)}3 \otimes 1^{(2)}3 )=(-1)^2 \big( \pi_{\nu}(1^{(3)}3 \otimes 2^{(2)}3) + \tbinom{2}{1}\pi_{\nu}(1^{(3)}2 \otimes 23^{(2)})  \big)
		\]
		and thus from Lemma \ref{insertrows} we get \begin{equation}\label{exeq2} \pi_{\mu}(12^{(2)}3 \otimes 1^{(2)}3 \otimes 23)=(-1)^2 \big( \pi_{\mu}(1^{(3)}3 \otimes 2^{(2)}3 \otimes 23) + \tbinom{2}{1}\pi_{\mu}(1^{(3)}2 \otimes 23^{(2)}  \otimes 23) \big). \end{equation}
		Similarly, for the second summand of the right hand side of eq. (\ref{exeq1}) we get
		\begin{equation}\label{exeq3} \pi_{\mu}(12^{(2)}3 \otimes 1^{(2)}2 \otimes 3^{(2)})=(-1)^2 \big( \tbinom{3}{1}\pi_{\mu}(1^{(3)}3 \otimes 2^{(3)} \otimes 3^{(2)}) + \tbinom{2}{1}\pi_{\mu}(1^{(3)}2 \otimes 2^{(2)}3  \otimes 3^{(2)})  \big). \end{equation}
		
		Substituting eqs. (\ref{exeq2}) and (\ref{exeq3}) in eq. (\ref{exeq1}) we find \begin{align}\label{exeq4}
		\pi_{\mu}(e^S)=&- \pi_{\mu}(1^{(3)}3 \otimes 2^{(2)}3 \otimes 23) -2 \pi_{\mu}(1^{(3)}2 \otimes 23^{(2)}  \otimes 23) \\\nonumber& - 6\pi_{\mu}(1^{(3)}3 \otimes 2^{(3)} \otimes 3^{(2)}) - 4 \pi_{\mu}(1^{(3)}2 \otimes 2^{(2)} 3 \otimes 3^{(2)}).
		\end{align}
		We note that in each summand of the right hand side of eq. (\ref{exeq4}), all the 1's are located in the first row.
		
		The third and fourth summands in the right hand side of eq. (\ref{exeq4}) are multiples of semistandard basis elements of the Weyl module $K_\mu$. However, the first and second summands in the right hand side of eq. (\ref{exeq4}) are multiples of $\pi_\mu(e^{S_1})$ and $\pi_\mu(e^{S_2})$ respectively, where \[S_1=\begin{matrix*}[l]
			1^{(3)}3 \\
			2^{(2)}3 \\
			23 \end{matrix*}, \ S_2 =\begin{matrix*}[l]
			1^{(3)}2 \\
			23^{(2)} \\
			23 \end{matrix*} \]
	and these tableaux are not semistandard because in both cases the 2 in the third row presents a violation of semistandardness. We may apply Lemma \ref{lemglas}(2) for the tableaux \[S_1[2,3]=\begin{matrix*}[l] 
		2^{(2)}3 \\
		23 \end{matrix*} \ \ \mathrm{and} \ \ S_2[2,3]=\begin{matrix*}[l] 
		23^{(2)} \\
		23 \end{matrix*} \]to raise the 2 from row 3 of $S_i$, $i=1,2$, to row 2. (In the notation of Lemma \ref{lemglas}, the 1's and 2's in the entries of the tableau are replaced by 2's and 3's respectively). Thus \begin{align*}\pi_\mu(e^{S_1})=&-2\pi_\mu(1^{(3)}3 \otimes 2^{(3)} \otimes 3^{(2)}), \\ \pi_\mu(e^{S_2})=&-2\pi_\mu(1^{(3)}3 \otimes 2^{(2)}3 \otimes 3^{(2)}) \end{align*}
	and the right hand sides are multiples of semistandard basis elements. Substituting in eq. (\ref{exeq4})	we obtain $\pi_\mu(e^S)$ as a linear combination of semistandard basis element of $K_\mu$. 
	\end{ex}
	
		\begin{rem}Let $\mu$ be a partition and $S$ a row semistandard tableau of shape $\mu$. We remark that by successive applications of Lemma \ref{lemglas} and Lemma \ref{insertrows} we may express $\pi_\mu(e^S)$ as a linear combination of various $\pi_\mu(e^{S\{j\}})$, where each tableau $S\{j\}$ has the property that all the 1's are located on the first row, all the 2's are located on the first two rows etc. This property, in general, does not imply that $S\{j\}$ is semistandard, for example $\begin{matrix*}[l] 
				124^{(2)} \\
				23^{(2)} \end{matrix*}$ is not semistandard.  However, it  turns out for the particular tableaux $S$ that appear in Section 3 (i.e. in the proofs of Lemmas \ref{lemma1}, \ref{lemma2}, \ref{lemma3}, \ref{lemma4} and \ref{gamma3}), the $S\{j\}$ obtained from $S$ by applying successively Lemma \ref{lemglas} and Lemma \ref{insertrows} are semistandard.  \end{rem}
\subsection{Projections} Since the characteristic of $\mathbb{K}$ is zero, every finite dimensional polynomial representation $M$ of $G$ is a direct sum of Weyl modules. The multiplicity of $K_\mu$, where $\mu \in \Lambda^+(N,r)$, as a summand of $M$ is equal to the dimension of the vector space $\Hom_G(M,K_\mu)$. When $M=D(\alpha)$, where $\alpha \in \Lambda(N,r)$, the dimension of $\Hom_G(D(\alpha),K_\mu)$, which is known as a Kostka number, is equal to the cardinality of the set $\mathrm{SST}_\alpha (\mu)$. In the sequel, we will need to identify different copies of $K_\mu$ in $D(\alpha)$. To this end we need an explicit basis of $\Hom_G(D(\alpha),K_\mu)$ which we describe next.

In what follows we will restrict our discussion to partitions that have at most three parts, since only such partitions are needed in the sequel. 

Suppose $\mu=(\mu_1,\mu_2,\mu_3) \in \Lambda^+(3,r)$, $\alpha=(\alpha_1,\alpha_2,\alpha_3) \in \Lambda(3,r)$ and $S$ is a row semistandard tableau of shape $\mu$ and weight $\alpha$
\[ S=\begin{matrix*}[l]
	1^{(a_{11})} 2^{(a_{12})} 3^{(a_{13})} \\
	1^{(a_{21})} 2^{(a_{22})} 3^{(a_{23})}\\
	1^{(a_{31})} 2^{(a_{32})} 3^{(a_{33})}. \end{matrix*}\]
This means that for the matrix $A=(a_{ij})$ the row sums are given by $\mu$ 
\[a_{i1}+a_{i2}+a_{i3}=\mu_i \ (i=1,2,3)\]
and the column sums by $\alpha$
\[a_{1j}+a_{2j}+a_{3j}=\alpha_j \ (j=1,2,3).\] We refer to $A$ as the matrix of the row semistandard tableau $S$. \begin{defn}\label{phiS}Suppose $\mu=(\mu_1,\mu_2,\mu_3) \in \Lambda^+(3,r)$, $\alpha=(\alpha_1,\alpha_2,\alpha_3) \in \Lambda(3,r)$ and $S$ is a row semistandard tableau of shape $\mu$ and weight $\alpha$. Let $A=(a_{ij})$ be the matrix of $S$. 
	
	\begin{enumerate} \item Define a map of $G$-modules \[\phi_S : D(\alpha) \to D{(\mu)}\] as the following composition
\begin{align}\label{phis} D(\alpha_1, \alpha_2, \alpha_3) \xrightarrow{\Delta_1\otimes \Delta_2 \otimes \Delta_3}& D(a_{11},a_{21},a_{31}) \otimes D(a_{12},a_{22},a_{32}) \otimes D(a_{13},a_{23},a_{33}) \\\nonumber \simeq &D(a_{11},a_{12},a_{13}) \otimes D(a_{21},a_{22},a_{23}) \otimes D(a_{31},a_{32},a_{33}) \\\nonumber\xrightarrow{m_1 \otimes m_2 \otimes m_3 }& D(\mu_1, \mu_2, \mu_3),
\end{align}
where $\Delta_i : D(a_{1i}+a_{2i}+a_{3i}) \to D(a_{1i}, a_{2i}, a_{3i})$ is the indicated component of twofold comultiplication of the Hopf algebra $D$,  the isomorphism permutes tensor factors, and $m_i $: $D(a_{i1}, a_{i2}, a_{i3}) \to D(a_{i1}+a_{i2}+a_{i3})$ is the indicated component of twofold multiplication in the algebra $D$. 
\item Define the map of $G$-modules \[\pi_S : D(\alpha) \to K_\mu\] as the composition \[\pi_S : D(\alpha) \xrightarrow{\phi_S} D(\mu) \xrightarrow{\pi_\mu} K_\mu.\]
\end{enumerate}\end{defn}

We note that if in the previous definition the matrix $A$ of the tableau $S$ is diagonal, then $\alpha = \mu$ and $\pi_S = \pi_\mu$. 

\begin{ex}\label{exphiS}
Suppose \[\mu = (5,4,1), \ \ \alpha = (3,6,1) \ \ \mathrm{and} \ \ S =\begin{matrix*}[l]
	1^{(2)} 2^{(3)} \\
	12^{(3)} \\
	3 \end{matrix*} \in \mathrm{RSST}_\alpha({\mu}).\] Then the matrix $A$ of the tableau $S$ is \[A= \begin{pmatrix} 2&3&0\\
	1&3&0\\0&0&1\end{pmatrix}.\] With the notation of Definition \ref{phiS}(1)  we have the comultiplication maps 
	\begin{align*}
		&\Delta_1 : D(3) \to D(2,1,0), \\
		&\Delta_2 : D(6) \to D(3,3,0), \\
		&\Delta_3 : D(1) \to D(0,0,1).\end{align*}
	Consider the map \[\phi_S: D(\alpha) \to D(\mu) \] and the element \[x=12^{(2)}\otimes 12^{(5)} \otimes 3 \in D(\alpha).\]
According to Definition \ref{phiS}(1)  we have \begin{align*}
	&\Delta_1(12^{(2)})=12 \otimes 2 + 2^{(2)}\otimes 1, \\
	&\Delta_2(12^{(5)})=12^{(2)} \otimes 2^{(3)}+2^{(3)} \otimes 12^{(2)}, \\
	&\Delta_3(3)=3.\end{align*} Hence the image of $x$ under the map $\phi_S:D(3,6,1) \to D(5,4,1)$ is equal to 
\begin{align*}
	\phi_S(x)&= \tbinom{1+1}{1} \tbinom{1+2}{1}\tbinom{1+3}{1} 1^{(2)}2^{(3)} \otimes  2^{(4)} \otimes 3 \\ &+ \Big( \tbinom{1+3}{1}\tbinom{1+2}{1}+\tbinom{2+2}{2} \Big) 12^{(4)}\otimes 12^{(3)} \otimes 3 +\tbinom{2+3}{2}\tbinom{1+1}{1} 2^{(5)}\otimes 1^{(2)}2^{(2)} \otimes 3.
\end{align*} 
The binomial coefficients come from the multiplication in the divided power algebra $D$.
\end{ex}

Similarly to Definition \ref{phiS}(1) we have a map for exterior powers in place of divided powers. \begin{defn}\label{defpsiS}Suppose $\mu=(\mu_1,\mu_2,\mu_3) \in \Lambda^+(3,r)$, $\alpha=(\alpha_1,\alpha_2,\alpha_3) \in \Lambda(3,r)$ and $S$ is a row semistandard tableau of shape $\mu$ and weight $\alpha$. Let $A=(a_{ij})$ be the matrix of $S$. Define a map of $G$-modules \[\psi_S : \Lambda(\alpha) \to \Lambda{(\mu)}\] as the following composition
\begin{align}\label{psiS} \Lambda(\alpha_1, \alpha_2, \alpha_3) \xrightarrow{\Delta_1\otimes \Delta_2 \otimes \Delta_3}& \Lambda(a_{11},a_{21},a_{31}) \otimes \Lambda(a_{12},a_{22},a_{32}) \otimes \Lambda(a_{13},a_{23},a_{33}) \\\nonumber \cong &\Lambda(a_{11},a_{12},a_{13}) \otimes \Lambda(a_{21},a_{22},a_{23}) \otimes \Lambda(a_{31},a_{32},a_{33}) \\\nonumber\xrightarrow{m_1 \otimes m_2 \otimes m_3 }& \Lambda(\mu_1, \mu_2, \mu_3),
\end{align}
where $\Delta_i$ is the indicated component of twofold comultiplication of the exterior algebra $\Lambda$,  the isomorphism permutes tensor factors, and $m_i$ is the indicated component of twofold multiplication in the exterior algebra $\Lambda$.\end{defn}

From \cite[Section 2, eq. (11)]{AB} we know the following for the maps $\pi_S : D(\alpha) \to K_\mu$ of Definition \ref{phiS}. \begin{prop}\label{wbasis} A basis of the vector space $\Hom_G(D(\alpha), K_\mu)$ is the set \begin{equation}
	\{\pi_S: S \in \mathrm{SST}_\alpha(\mu)\}.
\end{equation}\end{prop}

\begin{rem}\label{cyclic}
	It is well known that for every $\alpha \in \Lambda(N,r)$, the $G$-module $D(\alpha)$ is cyclic and a generator is the element \begin{equation}\label{ealpha} e^\alpha := 1^{(\alpha_1)} \otimes 2^{(\alpha_2)} \otimes \cdots \otimes N^{(\alpha_N)}.\end{equation} Hence the map $\pi_S : D(\alpha) \to K_\mu$ of Definition \ref{phiS} is determined by the image $\pi_S (e^\alpha)$ of $e^\alpha$.
\end{rem}

\subsection{The functor $\Omega$}\label{Omega}
Let us recall that there is an algebra involution on the ring of symmetric functions that sends the Schur function $s_\mu$ to $s_{\mu'}$ for every partition $\mu$ \cite[6.2]{F}. In terms of representations, we recall from \cite[p. 189]{AB} that, since the characteristic of $\mathbb{K}$ is zero, there is an involutory natural equivalence $\Omega$ from the category of polynomial representations of $G$ of degree r, where $N \ge r$, to itself that has the following properties.  \begin{enumerate} \item $\Omega(D(\alpha)) = \Lambda (\alpha)$ for all $\alpha \in \Lambda(N,r)$ and  $\Omega (K_\mu)=L_\mu$ for every $\mu \in \Lambda^+(N,r)$. More generally, if $\mu(1) \in \Lambda^{+}(N,r_1), \dots, \mu(q) \in \Lambda^{+}(N,r_q)$ are partitions such that $r_1 + \cdots + r_q =r$, then \[\Omega(K_{\mu(1)} \otimes \cdots \otimes K_{\mu(q)} ) = L_{\mu(1)} \otimes \cdots \otimes L_{\mu(q)}.\]\item The functor $\Omega$ preserves the comultiplication and multiplication maps of the Hopf algebras $D$ and $\Lambda$. 
	
	To be precise, this means that for all $(\alpha_1, \dots, \alpha_N) \in \Lambda(N,r)$ and all $s$ the images under $\Omega$ of the maps \begin{align*}&1\otimes \cdots \otimes \Delta_D \otimes \cdots \otimes 1 :D(\alpha_1, \dots, \alpha_s, \dots, \alpha_N) \to D(\alpha_1, \dots, \alpha'_s, \alpha''_s \dots, \alpha_N), \\&
		1\otimes \cdots \otimes \eta_D \otimes \cdots \otimes 1 :D(\alpha_1, \dots, \alpha_s, \alpha_{s+1} \dots, \alpha_N) \to D(\alpha_1, \dots, \alpha_s +\alpha_{s+1} \dots, \alpha_N)
		\end{align*}
are the maps
\begin{align*}&1\otimes \cdots \otimes \Delta_{\Lambda} \otimes \cdots \otimes 1 :\Lambda(\alpha_1, \dots, \alpha_s, \dots, \alpha_N) \to \Lambda(\alpha_1, \dots, \alpha'_s, \alpha''_s \dots, \alpha_N), \\&
	1\otimes \cdots \otimes \eta_{\Lambda} \otimes \cdots \otimes 1 :\Lambda(\alpha_1, \dots, \alpha_s, \alpha_{s+1} \dots, \alpha_N) \to \Lambda(\alpha_1, \dots, \alpha_s +\alpha_{s+1} \dots, \alpha_N)
\end{align*}
respectively. Here, $\alpha_s = \alpha'_s +\alpha''_s$ and $\Delta_D :D_{\alpha_s} \to D_{\alpha'_s} \otimes D_{\alpha''_s}$ 	and $\eta_D: D_{\alpha_s} \otimes D_{\alpha_{s+1}} \to D_{\alpha_s +\alpha_{s+1}}$
are the indicated components of the comultiplication and multiplication maps of the divided power algebra $D$ respectively. Likewise, $\Delta_\Lambda :\Lambda^{\alpha_s} \to \Lambda^{\alpha'_s} \otimes \Lambda^{\alpha''_s}$ 	and $\eta_\Lambda: \Lambda^{\alpha_s} \otimes \Lambda^{\alpha_{s+1}} \to \Lambda^{\alpha_s +\alpha_{s+1}}$
are the indicated components of the comultiplication and multiplication maps of the exterior algebra $\Lambda$ respectively.	
	
\item $\Omega(\tau_{s,D}) = (-1)^{\alpha_s \alpha_{s+1}}\tau_{s,\Lambda}$ for all $\alpha =(\alpha_1, \dots, \alpha_N) \in \Lambda(N,r)$ and all $s$, where $\tau_{s,D}: D(\alpha) \to D(\alpha)$ (respectively, $\tau_{s,\Lambda}: \Lambda(\alpha) \to \Lambda(\alpha)$) is the map that interchanges the factors $D_{\alpha_s}$ and $D_{\alpha_{s+1}}$ (respectively, $\Lambda^{\alpha_s}$ and $\Lambda^{\alpha_{s+1}}$) and is the identity on the rest.\item The functor $\Omega$ is exact.\end{enumerate}

\section{A result for Weyl modules}\label{sec3}
Suppose $n \ge 2$. Throughout this section,  $\lambda$, $\mu$ and $\nu$ are the following partitions of $3n-2$
\begin{align*}\lambda&:=(n,n-1,n-1), \\ \mu&:=(n+1,n-1,n-2), \\ \nu&:=(n,n,n-2).\end{align*}
This section contains some important results on Weyl modules whose proofs  are quite technical. We give here a summary of the overall picture. We define certain maps $\gamma_1, \gamma_2  \in \Hom _G(D(\lambda), D(\lambda))$ and $\gamma_3  \in \Hom _G(D(\nu), D(\lambda))$ (see Definition \ref{mapg} and Definition \ref{mapg3}). The main result of Section 3 is the determination of the cokernel of the map \[\gamma_1 + \gamma_2 + \gamma_3 : D(\lambda) \oplus D(\lambda) \oplus D(\nu) \to D(\lambda),\] where $(\gamma_1+\gamma_2+\gamma_3)(x,y,z)=\gamma_1(x)+\gamma_2(y)+\gamma_3(z)$ for $x,y \in D(\lambda)$ and $z \in D(\nu)$ (see Theorem \ref{main2}). This is done in a few steps.

(a) Let us consider the map $\gamma_1+\gamma_2 : D(\lambda) \oplus D(\lambda) \to D(\lambda)$, where $(\gamma_1+\gamma_2)(x,y)=\gamma_1(x)+\gamma_2(y)$ for $x,y \in D(\lambda)$. We show that the composition $D(\lambda) \xrightarrow{\gamma_i} D(\lambda) \xrightarrow{\pi_\lambda} K_\lambda$ is the zero map for $i=1,2$, from which it follows that $K_\lambda$ is not a summand of  $\Ima(\gamma_1+\gamma_2)$ (see Lemma \ref{lemma1}).

(b) Let $\{\pi_{R_1}, \pi_{R_2}\}$ be the basis of $\Hom_G(D(\lambda), K_\mu)$ given by Proposition \ref{wbasis}. First we show that the maps 
$D(\lambda) \xrightarrow{\gamma_1} D(\lambda) \xrightarrow{\pi_{R_j}} K_\mu$ for $i=1,2$ are nonzero and linearly dependent. Next we show that each of the maps $D(\lambda) \xrightarrow{\gamma_2} D(\lambda) \xrightarrow{\pi_{R_j}} K_\mu$ is the zero map for $i=1,2$. From these results it follows that the multiplicity of $K_\mu$ in $\Ima(\gamma_1+\gamma_2)$ is equal to $1$ (see Lemma \ref{lemma2}).

(c) In Lemmas \ref{lemma3} and \ref{lemma4} we prove a combinatorial property of irreducible summands of  the tensor product $K_{(n,n-1)} \otimes D_{n-1}$ under the actions of the maps $\gamma_1$ and ${\gamma_2}$ respectively. From this it follows that the multiplicity in $\Ima(\gamma_1+\gamma_2)$ of $K_\xi$ is equal to zero for any summand $K_\xi$ of $K_{(n,n-1)} \otimes D_{n-1}$ such that $\xi \neq \lambda, \mu$ (see the proof of Theorem \ref{main1}). Combining this and steps (a), (b) we obtain that $\coker(\gamma_1+\gamma_2) \cong K_\lambda \oplus K_\mu$ (see Theorem \ref{main1}).

(d) With the notation of step (b), we show that each of the maps  $D(\nu) \xrightarrow{\gamma_3} D(\lambda) \xrightarrow{\pi_{\lambda}} K_\lambda$ and $D(\nu) \xrightarrow{\gamma_3} D(\lambda) \xrightarrow{\pi_{R_i}} K_\mu$ is the zero map, where $i=1,2$ (see Lemma \ref{gamma3}). From this and step (c) it follows that $\coker(\gamma_1+\gamma_2+\gamma_3) \cong K_\lambda \oplus K_\mu$ (see Theorem \ref{main2}).
\subsection{The maps $\gamma_1$ and $\gamma_2$}
\begin{defn}\label{mapg} Define $\gamma_1, \gamma_2  \in \Hom _G(D(\lambda), D(\lambda))$ by \begin{align*}\gamma_1 &:= \phi_{S(1)}+ (-1)^n\phi_{S(2)} \\ \gamma_2 &:= \phi_{S(1)}+(-1)^n \phi_{S(3)} +(-1)^n \phi_{S(4)},\end{align*} where the tableaux $S(i) \in \mathrm{RSST}_\lambda(\lambda) $ are the following
	\[ S(1):=\begin{matrix*}[l]
		1^{(n)}  \\
		2^{(n-1)} \\
		3^{(n-1)} \end{matrix*}, \; S(2):=\begin{matrix*}[l]
		12^{(n-1)}  \\
		1^{(n-1)} \\
		3^{(n-1)} \end{matrix*}, \; S(3):=\begin{matrix*}[l]
		1^{(n)}  \\
		3^{(n-1)} \\
		2^{(n-1)} \end{matrix*}, \; S(4):=\begin{matrix*}[l]
		23^{(n-1)}  \\
		12^{(n-2)} \\
		1^{(n-1)} \end{matrix*}. \]	
\end{defn}

\begin{rem}\label{gamma1} Concerning the map $\gamma_1$ defined above, we observe that, according to Definition \ref{phiS}, $\phi_{S(1)} :D(\lambda) \to D(\lambda)$ is the identity map and $\phi_{S(2)} :D(\lambda) \to D(\lambda)$ is the composition \begin{align*}  D_n \otimes D_{n-1} \otimes D_{n-1} &\xrightarrow {\Delta \otimes 1 \otimes 1} (D_1 \otimes D_{n-1}) \otimes D_{n-1} \otimes D_{n-1} \\&\xrightarrow{1 \otimes \tau \otimes 1}  D_1 \otimes (D_{n-1} \otimes D_{n-1}) \otimes D_{n-1} \\&\xrightarrow{\eta \otimes 1 \otimes 1}  D_n \otimes D_{n-1} \otimes D_{n-1}, \end{align*}
	where $\Delta: D_{n} \to D_1 \otimes D_{n-1}$ (respectively, $\eta: D_{1} \otimes D_{n-1} \to D_{n}$) is the indicated component of the comultiplication map (respectively, multiplication map) of the divided power algebra $D$, and $\tau: D_{n-1} \otimes  D_{n-1} \to D_{n-1} \otimes  D_{n-1}$ is the map defined by $\tau(x \otimes y) =y \otimes x$, for $x, y \in D_{n-1} $.
\end{rem}

Recall from eq. (\ref{ealpha}) the notation $e^{\lambda}=1^{(n)} \otimes 2^{(n-1)} \otimes 3^{(n-1)}$.

\begin{lem}\label{ge} With the notation of Definition \ref{mapg}, let $\epsilon = (-1)^n$. Then \begin{align}\label{gammae}&\gamma_1(e^{\lambda})= e^{\lambda} +\epsilon12^{(n-1)} \otimes 1^{(n-1)} \otimes 3^{(n-1)}, \\ \label{phie}&\gamma_2(e^{\lambda})= e^{\lambda} +\epsilon1^{(n)} \otimes 3^{(n-1)} \otimes 2^{(n-1)}+\epsilon23^{(n-1)} \otimes 12^{(n-2)} \otimes 1^{(n-1)}.\end{align} \end{lem}
\begin{proof}Both equations follow from Definition \ref{mapg} and Definition \ref{phiS}(1).
	\end{proof}

The motivation for considering such maps will become clear in sections 5.1 and 5.2. Roughly speaking, the images of these maps correspond to certain relations of the multilinear component $\text{Lie}_n(m)$ of the free LAnKe that are consequences of the generalized Jacobi identity. 

We intend to prove the following theorem in Section \ref{proofmain}, which is the first main result of Section \ref{sec3} . Consider the maps $\gamma_1, \gamma_2 : D(\lambda) \to D(\lambda)$ and recall we have the map  $\gamma_1 +\gamma_2 : D(\lambda) \oplus D(\lambda) \to D(\lambda)$ defined by $(\gamma_1 +\gamma_2)(x,y) =\gamma_1(x) + \gamma_2(y)$, where $x,y \in D(\lambda)$. It follows that $\Ima(\gamma_1 +\gamma_2)=\Ima(\gamma_1)+\Ima(\gamma_2)$. \begin{thm}\label{main1}
	Let $N \ge 3n-2$. Then, the cokernel of the map \[\gamma_1 +\gamma_2 : D(\lambda) \oplus D(\lambda) \to D(\lambda) \] is isomorphic to $K_{\lambda} \oplus K_{\mu}$ as $G$-modules.
\end{thm}

\subsection{The cokernel of the map $\gamma_1: D(\lambda) \to D(\lambda)$}
For the proof of Theorem \ref{main1} we will need to identify the cokernel of the map $\gamma_1: D(\lambda) \to D(\lambda)$. This may be done using a result from \cite{MMS} which we now recall.

\begin{defn}\cite[Definition 5.1]{MMS} \begin{enumerate}\item Let $\beta_{n-1}$ be the map of tensor product of exterior powers \[\beta_{n-1}: \Lambda^n \otimes \Lambda^{n-1} \to \Lambda^n \otimes \Lambda^{n-1}\] given by the composition
\[  \Lambda^n \otimes \Lambda^{n-1} \xrightarrow {\Delta \otimes 1} \Lambda^1 \otimes \Lambda^{n-1} \otimes \Lambda^{n-1} \xrightarrow{\tau}  \Lambda^1 \otimes \Lambda^{n-1} \otimes \Lambda^{n-1} \xrightarrow{\eta \otimes 1}  \Lambda^{n} \otimes \Lambda^{n-1}, \] 
where $\Delta: \Lambda^{n} \to \Lambda^1 \otimes \Lambda^{n-1}$ (respectively, $\eta: \Lambda^{1} \otimes \Lambda^{n-1} \to \Lambda^{n}$) is the indicated component of the comultiplication map (respectively, multiplication map) of the exterior algebra $\Lambda$, and $\tau: \Lambda^{n-1} \otimes  \Lambda^{n-1} \to \Lambda^{n-1} \otimes  \Lambda^{n-1}$ is the defined by $\tau(x \otimes y) =y \otimes x$, for $x, y \in \Lambda^{n-1} $. \item Let $g_{n-1}$ be the map \[g_{n-1}: \Lambda^n \otimes \Lambda^{n-1} \to \Lambda^n \otimes \Lambda^{n-1}, \ g_{n-1}(x\otimes y) =x \otimes y - \beta_{n-1}(x \otimes y). \] \end{enumerate}\end{defn}
To be precise, the maps $\beta_{n-1}$ and $g_{n-1}$ above are the special cases of the maps $\beta_k$ and $\gamma_k$ defined in \cite[Definition 5.1]{MMS} for $a=n$ and $ b=k=n-1$ in the notation of loc. cit. Our $g_{n-1}$ is denoted  $\gamma_{n-1}$ in loc. cit. The $\gamma_{n-1}$ of \cite{MMS} should not be confused with the maps given in Definition \ref{mapg} of the present paper.

We recall the following special case of \cite[Corollary 5.4]{MMS}. Here $L_{(n,n-1)}$ denotes the Schur module corresponding to the partition $(n, n-1)$ (see Section 2.2).
\begin{lem}\label{cokg} Suppose $N \ge 2n-1$. Then $\coker(g_{n-1}) \simeq L_{(n,n-1)}$.
\end{lem}
The main idea of the proof of this lemma given in \cite{MMS} is the computation of the eigenvalues of the map $g_{n-1}: \Lambda^n \otimes \Lambda^{n-1} \to \Lambda^n \otimes \Lambda^{n-1}$ on the irreducible summands of $\Lambda^n \otimes \Lambda^{n-1}$. This was accomplished with the use of combinatorics of tableaux and in particular the straightening law.

Now we may identify the cokernel of the map $\gamma_1: D(\lambda) \to D(\lambda)$ of Definition \ref{mapg}.
\begin{lem}\label{lem37} Suppose $N \ge 3n-2$. Then $\coker(\gamma_1) \simeq K_{(n,n-1)} \otimes D(n-1)$.
\end{lem}
\begin{proof}
	Consider the involutive functor $\Omega$ of Section \ref{Omega}. Using properties (1) - (3) of $\Omega$, it follows from Remark \ref{gamma1}, that the image of the map  \[g_{n-1} \otimes 1: \Lambda^n \otimes \Lambda^{n-1} \otimes \Lambda^{n-1} \to \Lambda^n \otimes \Lambda^{n-1}\otimes \Lambda^{n-1}\] under $\Omega$ is the map $\gamma_1: D(n,n-1,n-1) \to D(n,n-1,n-1)$. Since $\Omega$ is an exact functor, we have  \[\coker{\gamma_1} \simeq \Omega (\coker{g_{n-1} \otimes 1}) \simeq \Omega (L_{(n,n-1)} \otimes \Lambda^{n-1}) \simeq K_{(n,n-1)} \otimes D_{n-1},\]
	where the middle isomorphism is due to Lemma \ref{cokg}.
\end{proof}

Consider the map \[\gamma_1 +\gamma_2 : D(\lambda) \oplus D(\lambda) \to D(\lambda) \] in the statement of Theorem \ref{main1}. An immediate consequence of Lemma \ref{lem37} is the following. 
\begin{cor}\label{cor38}
	Suppose $N \ge 3n-2$. Then the $G$-module $\coker{(\gamma_1+\gamma_2)}$ is a quotient of $K_{(n,n-1)} \otimes D_{n-1}$.
\end{cor}

We note without pursuing details that  another proof of Lemma \ref{lem37} may be obtained using \cite[Theorem 1.3]{FHSW} by applying first the `inverse' Schur functor \cite[pg. 56]{Gr} and then the functor $\Omega$.
\subsection{The actions of the maps $\gamma_1$ and $\gamma_2$}
The next four lemmas analyze the action of the map $\gamma_1 +\gamma_2 : D(\lambda) \oplus D(\lambda) \to D(\lambda)$ on the irreducible summands of $D(\lambda)$. By Corollary \ref{cor38}, we need only consider those irreducible summands of $D(\lambda)$ that are summands of the module $K_{(n,n-1)} \otimes D_{n-1}$.

As mentioned at the beginning of Section 3, the first two lemmas concern the summands $K_{\lambda}$ and $K_{\mu}$. We show in the last two lemmas that the other possible irreducible summands have a combinatorial property that will be utilized in Section 3.4 to prove that their multiplicities in $\coker(\gamma_1 +\gamma_2) $ are in fact zero.
\subsection*{Multiplicities of the irreducibles $K_\lambda$ and $K_{\mu}$ in $\coker(\gamma_1+\gamma_2)$}
\begin{lem}\label{lemma1} The multiplicity of $K_{\lambda}$ in $\coker{(\gamma_1 + \gamma_2)}$ is equal to one.
\end{lem}
\begin{proof}
	Let $T_0=\pi_{\lambda}(e^{\lambda})$. First we show that the composition $\pi_{\lambda} \circ \gamma_1$ is the zero map. Using eq. (\ref{gammae}) we have \[\pi_{\lambda} \circ \gamma_1(e^{\lambda})=T_0+(-1)^n\pi_{\lambda}(12^{(n-1)} \otimes 1^{(n-1)} \otimes 3^{(n-1)}).\] Applying Lemma \ref{lemglas}(2) to raise the 1's from row 2 to row 1, we have \[\pi_{\lambda}(12^{(n-1)} \otimes 1^{(n-1)} \otimes 3^{(n-1)})=(-1)^{n-1}\pi_{\lambda}(1^{(n)} \otimes 2^{(n-1)} \otimes 3^{(n-1)})=(-1)^{n-1}T_0.\] By substituting we obtain \[\pi_{\lambda} \circ \gamma_1(e^{\lambda}) = T_0 +(-1)^{2n-1}T_0=0.\]
	Since $e^\lambda$ generates the $G$-module $D(\lambda)$ (according to Remark \ref{cyclic}) and $\pi_\mu, \gamma_1$ are maps of $G$-modules, the above equation yields $\pi_{\lambda} \circ \gamma_1=0$.
	
	Next we show that the composition $\pi_{\lambda} \circ \gamma_2$ is the zero map. Using eq. (\ref{phie}) we have \begin{equation}\label{341}\pi_{\lambda} \circ \gamma_2(e^{\lambda})=T_0+(-1)^n\pi_{\lambda}(1^{(n)} \otimes 3^{(n-1)} \otimes 2^{(n-1)}) +  (-1)^n\pi_{\lambda}(23^{(n-1)} \otimes 12^{(n-2)} \otimes 1^{(n-1)}).\end{equation}
	We apply to the second summand in the right hand side of eq. (\ref{341}) Lemma \ref{lemglas}(2) (to raise the 2's from row 3 to row 2) obtaining
	\[\pi_{\lambda}(1^{(n)} \otimes 3^{(n-1)} \otimes 2^{(n-1)})= (-1)^{n-1}\pi_{\lambda}(1^{(n)} \otimes 2^{(n-1)} \otimes 3^{(n-1)})=(-1)^{n-1}T_0.\]
	For the third summand in the right hand side of eq. (\ref{341}) we have \[\pi_{\lambda}(23^{(n-1)} \otimes 12^{(n-2)} \otimes 1^{(n-1)})=0\]
because of Lemma \ref{lemglas}(1) applied to rows 2 and 3 (where the number of 1's is equal to $1+n-1 = n$ which is greater than the length $n-1$ of the second row.)

 Thus $K_{\lambda}$ is not a summand of the image $\Ima(\gamma_1 + \gamma_2)$. Since  the multiplicity of $K_{\lambda}$ in the codomain $D({\lambda})$ of the map $\gamma_1 +\gamma_2 : D(\lambda) \oplus D(\lambda) \to D(\lambda)$ is equal to 1, we conclude that the multiplicity of $K_{\lambda}$ in $\coker{(\gamma_1 + \gamma_2)}$ is equal to one.
\end{proof}

\begin{lem}\label{lemma2} The multiplicity of $K_{\mu}$ in $\coker{(\gamma_1 + \gamma_2)}$ is equal to one.
\end{lem}
\begin{proof}
	Recall that $\mu = (n+1,n-1,n-2)$. There are exactly two semistandard tableaux of shape $\mu$ and weight $\lambda$, 
	
	\[ R_1:=\begin{matrix*}[l]
		1^{(n)}3  \\
		2^{(n-1)} \\
		3^{(n-2)} \end{matrix*}, \;R_2:=\begin{matrix*}[l]
		1^{(n)}2  \\
		2^{(n-2)}3 \\
		3^{(n-2)} \end{matrix*}.  \]
	According to Definition \ref{phiS} and Proposition \ref{wbasis}, the corresponding projections are 	\begin{align*} &\pi_{R_1}:D(\lambda) \to K_\mu, e^{\lambda} \mapsto \pi_\mu(1^{(n)}3\otimes2^{(n-1)} \otimes 3^{(n-2)}), \\&\pi_{R_2}:D(\lambda) \to K_\mu, e^{\lambda} \mapsto \pi_\mu(1^{(n)}2\otimes2^{(n-2)}3 \otimes 3^{(n-2)}).\end{align*}
	We compute the maps $\pi_{R_i} \circ \gamma_1, \ \pi_{R_i} \circ \gamma_2 \in \Hom_G(D(\lambda), K_\mu)$. 
	
	For notational convenience, let $T_i:=\pi_{R_i}(e^\lambda)$, $i=1,2$. Using eq. (\ref{gammae}) of Lemma \ref {ge}  we have
	\[
		\pi_{R_1} \circ \gamma_1 (e^\lambda)=T_1+(-1)^n\pi_{\mu}(12^{(n-1)}3 \otimes 1^{(n-1)} \otimes 3^{(n-2)}).
	\]
	We apply Lemma \ref{lemglas}(2) to the second summand of the right hand side to obtain
	\begin{align*}
		(-1)^n\pi_{\mu}(12^{(n-1)}3 \otimes 1^{(n-1)} \otimes 3^{(n-2)}&=(-1)^{2n-1}\big(\pi_{\mu}(1^{(n)}3 \otimes 2^{(n-1)} \otimes 3^{(n-2)}) \\& \;\;\;\; + \pi_{\mu}(1^{(n)}2 \otimes 2^{(n-2)}3 \otimes 3^{(n-2)}) \big)\\&=-T_1-T_2.
	\end{align*}
	Thus \[\pi_{R_1} \circ \gamma_1 (e^\lambda)=-T_2\]

	In a similar manner we have \begin{align*}	\pi_{R_2} \circ \gamma_1 (e^\lambda)&=T_2+(-1)^n\tbinom{2}{1} \pi_{\mu}(1^{(2)}2^{(n-1)} \otimes 1^{(n-2)}3 \otimes 3^{(n-2)})\\&=T_2+(-1)^{2n-2}\tbinom{2}{1}T_2 \\&=3T_2.\end{align*}
	
	Since $e^\lambda$ generates the $G$-module $D(\lambda)$, we conclude from the above that the restrictions of the maps $\pi_{R_1}, \pi_{R_2} \in \Hom_G(D(\lambda), K_\mu)$ to the image $\Ima(\gamma_1) \subseteq D(\lambda)$ are linearly dependent. 
	
	We show below that the restrictions of  $\pi_{R_1}, \pi_{R_2}$ to the image $\Ima(\gamma_2) \subseteq D(\lambda)$ are both zero. This implies that the restrictions of $\pi_{R_1}, \pi_{R_2}$ to $\Ima(\gamma_1 +\gamma_2)$ are linearly dependent. By the above computation, these restrictions are nonzero since $T_2 \neq 0$. Thus the multiplicity of $K_\mu$ in $\Ima(\gamma_1 +\gamma_2)$ is equal to 1. Finally, the multiplicity of  $K_\mu$ in $\coker(\gamma_1+\gamma_2)$ is equal to $2-1=1$.
	
	It remains to be shown that the restrictions of  $\pi_{R_1}, \pi_{R_2}$ to the image $\Ima(\gamma_2) \subseteq D(\lambda)$ are both zero. To this end, we first show the identity \begin{equation}\label{10}
		\pi_\mu(123^{(n-1)} \otimes 12^{(n-2)} \otimes 1^{(n-2)})=(-1)^{n-1}(T_1+T_2 ).
	\end{equation} The idea is to raise all the 1's to the first row and all the 2's to the first two rows. This will be done applying Lemma \ref{lemglas}(2) several times. 
	We start by applying \ref{lemglas}(2) to rows 2 and 3, 
\[	\pi_\mu(123^{(n-1)} \otimes 12^{(n-2)} \otimes 1^{(n-2)})=(-1)^{n-2}\pi_\mu(123^{(n-1)} \otimes 1^{(n-1)} \otimes 2^{(n-2)}).\]
Now we apply the same Lemma to rows 1 and 2 
\[(-1)^{n-2}\pi_\mu(123^{(n-1)} \otimes 1^{(n-1)} \otimes 2^{(n-2)})=- \pi_\mu(1^{(n)}2 \otimes 3^{(n-1)} \otimes 2^{(n-2)})-\pi_\mu(1^{(n)}3 \otimes 23^{(n-2)} \otimes 2^{(n-2)} ).\]	
We apply Lemma \ref{lemglas}(2) to each summand in the right hand side to rows 2 and 3 (to raise the 2's from the third row to the second) and we obtain 
\[(-1)^{n-2}\pi_\mu(123^{(n-1)} \otimes 1^{(n-1)} \otimes 2^{(n-2)})=(-1)^{n-1}(T_1+T_2).\]	
Thus we have show eq. (\ref{10}).

	Now using (\ref{phie}) we obtain \begin{align*}\pi_{R_1} \circ \gamma_2 (e^{\lambda}) &= T_1 +(-1)^n \pi_\mu(1^{(n)}2 \otimes 3^{(n-1)} \otimes 2^{(n-2)}) \\& \;\;\;\; +(-1)^n\pi_\mu(123^{(n-1)} \otimes 12^{(n-2)} \otimes 1^{(n-2)}) \\&= T_1+(-1)^{n+n-2}T_2+(-1)^{n+n-1}(T_1+T_2) \\&=0,
	\end{align*}
	where in the second equality we used once again Lemma \ref{lemglas}(2) and (\ref{10}).
	
	By a similar computation, we have
	\begin{align*}\pi_{R_2} \circ \gamma_2 (e^{\lambda}) &= T_2 +(-1)^n \pi_\mu(1^{(n)}3 \otimes 23^{(n-2)} \otimes 2^{(n-2)}) \\& \;\;\;\; +(-1)^n \Big(\pi_\mu(123^{(n-1)} \otimes 12^{(n-2)} \otimes 1^{(n-2)})  \\& \;\;\;\; +\tbinom{2}{1}^2 \pi_\mu(2^{(2)}3^{(n-1)} \otimes 1^{(2)}2^{(n-3)} \otimes 1^{(n-2)})\Big) \\&= T_2+(-1)^{n+n-2}T_1+(-1)^{n+n-1}(T_1+T_2 +\tbinom{2}{1}^20) \\&=0,
	\end{align*}
	where we used $\pi_\mu(2^{(2)}3^{(n-1)} \otimes 1^{(2)}2^{(n-3)} \otimes 1^{(n-2)})=0$ according to  Lemma \ref{lemglas}(1) applied to rows 2 and 3.
\end{proof}
\subsection*{Two combinatorial lemmas}
The next two lemmas concern a certain combinatorial property of irreducible summands of the tensor product $K_{(n,n-1)} \otimes  D_{n-1}$ under the actions of the maps $\gamma_1$ and $\gamma_2$. 

We need some notation. Let $\mathrm{Par}(K_{(n,n-1)} \otimes  D_{n-1})$ be the subset of $\Lambda(N, 3n-2)$ consisting of the partitions $\xi$ such that $K_\xi$ is a summand of $K_{(n,n-1)} \otimes  D_{n-1}$. By Pieri's rule, for example see \cite[Corollary 2.3.5]{W}, we have that $\mathrm{Par}(K_{(n,n-1)} \otimes  D_{n-1})$ consists of those partitions $\xi$ that are of the form \begin{equation}\label{mu}\xi=(n + c_1, n-1 + c_2, c_3),\end{equation}
for some nonnegative integers $c_1, c_2, c_3$ satisfying $c_1 +c_2 +c_3 = n-1$ and $c_2 \in \{0,1\}$.

Let us fix a partition $\xi=(n + c_1, n-1 + c_2, n-1-c_1-c_2)$ as in (\ref{mu}). One easily verifies that for each $i=0, \dots, c_1$, there is a unique semistandard tableau $U_i \in \mathrm{SST}_{\lambda}(\xi)$ such that the number of  2's in the first row is equal to $i$, and moreover \[\mathrm{SST}_\lambda(\xi) = \{U_0, U_1, \dots, U_{c_1} \}.\] We have 
\begin{equation}\label{Si} U_i=\begin{matrix*}[l]
		1^{(n)}2^{(i)}3^{(c_1-i)} \\
		2^{(n-1-i)}3^{(c_2+i)} \\
		3^{(n-1-c_1-c_2)}. \end{matrix*}\end{equation} 
Let \[Z_i :=\pi_\xi(e^{U_i}) \in K_\xi, \ \ i=0, \dots, c_1.\]  Since $\mathrm{SST}_\lambda(\xi) = \{U_0, U_1, \dots, U_{c_1} \}$, we know from Theorem \ref{sbthm} that the set \[\{Z_0, Z_1, \dots, Z_{c_1}\}\] is a basis of the Weyl module $K_\xi.$

The next lemma concerns the action of $\gamma_1$ on $K_\xi$.

\begin{lem}\label{lemma3}Let $\xi \in \mathrm{Par}(K_{(n,n-1)} \otimes  D_{n-1})$. Then for every semistandard $U_i \in \mathrm{SST}_{\lambda}(\xi)$ , the coefficient of $Z_0=\pi_\xi(e^{U_0}) \in K_\mu$ in the expression of  $\pi_{U_i} \circ \gamma_1 (e^\lambda) \in K_\xi$ as a linear combination of the basis elements $Z_0, Z_1, \dots, Z_{c_1}$ is equal to $0$.
\end{lem}
\begin{proof} Let $\xi=(n + c_1, n-1 + c_2, n-1-c_1-c_2)$ according to (\ref{mu}). For notational convenience, let $c=c_1+c_2$. 
	
We compute $\pi_{U_i} \circ \gamma_1 (e^\lambda)$ as a linear combination of the basis elements $Z_0, Z_1, \dots, Z_{c_1}$ of $K_\xi$.

First, using eq. (\ref{mapg}) and the definition of the map $\pi_{U_i}$ given in Definition \ref{phiS} we have 
\begin{align}\label{eq36}\pi_{U_i} \circ \gamma_1 (e^\lambda)&=\pi_{U_i}(e^{\lambda})+(-1)^n\pi_{U_i}(12^{(n-1)}\otimes1^{n-1}\otimes3^{n-1})\\\nonumber&=\pi_{U_i}(e^{\lambda})+ (-1)^n\pi_{\xi}(11^{(i)}2^{(n-1)}3^{(c_1-i)} \otimes 1^{(n-1-i)}3^{(c_2+i)} \otimes 3^{(n-1-c)})\\\nonumber&=Z_i+(-1)^n\tbinom{i+1}{1}\pi_{\xi}(1^{(i+1)}2^{(n-1)}3^{(c_1-i)} \otimes 1^{(n-1-i)}3^{(c_2+i)} \otimes 3^{(n-1-c)}),\end{align} where the binomial coefficient $\tbinom{i+1}{1}$ comes from the multiplication $11^{(i)}=\tbinom{i+1}{1}1^{(i+1)}$ in the divided power algebra $D$.
Next we apply Lemma \ref{lemglas}(2) to rows 1 and 2 of \[X:=\pi_{\xi}(1^{(i+1)}2^{(n-1)}3^{(c_1-i)} \otimes 1^{(n-1-i)}3^{(c_2+i)} \otimes 3^{(n-1-c)})\] to obtain 
	\begin{align*}X=(-1)^{n-1-i}\Big(&\pi_\xi(1^{(n)}2^{(i)}3^{(c_1-i)}\otimes 2^{(n-1-i)}3^{(c_2+i)}\otimes 3^{(n-1-c)} ) \\&+\tbinom{c_2+i+1}{1}\pi_\xi(1^{(n)}2^{(i+1)}3^{(c_1-i-1)}\otimes 2^{(n-2-i)}3^{(c_2+i-1)}\otimes 3^{(n-1-c)} \\&+ \cdots \\&+ \tbinom{c_1+c_2}{c_1-i}\pi_\xi(1^{(n)}2^{(c_1)}\otimes 2^{(n-1-c_1)}3^{(c_2+c_1)}\otimes 3^{(n-1-c)} ) \Big).\end{align*}
Thus we have \[X=(-1)^{n-1-i}\big( Z_i +\tbinom{c_2+i+1}{1}Z_{i+1} + \cdots + \tbinom{c_1+c_2}{c_1-i} Z_{c_1}\big)\]	and substituting this in eq. (\ref{eq36}) we obtain 
	
	\begin{equation}\label{eq362}\pi_{U_i} \circ \gamma_1 (e^\lambda)=Z_i+(-1)^{i+1}\tbinom{i+1}{1}\big( Z_i +\tbinom{c_2+i+1}{1}Z_{i+1} + \cdots + \tbinom{c_1+c_2}{c_1-i}Z_{c_1} \big).\end{equation}
	Now we see that if $i>0$, then the coefficient of $Z_0$ in the right hand side of eq. (\ref{eq362}) is equal to 0. Also, if $i=0$, then the coefficient of $Z_0$ in the right hand side of eq. (\ref{eq362}) is equal to $1+(-1)\tbinom{1}{1}=0$.
\end{proof}

We keep the previous notation, namely \begin{itemize}
	\item we have a partition $\xi  \in \mathrm{Par}(K_{(n,n-1)} \otimes  D_{n-1})$,
	\item we have the semistandard tableaux $U_i \in \mathrm{SST}_{\lambda} (\xi), i=0,1,\dots,c_1$, given by eq. (\ref{Si}),
	\item and we define $Z_i:=\pi_\xi(U_i)$, $i=0,1,\dots,c_1,$ which form a basis of the Weyl module $K_\xi$. 
\end{itemize}The next lemma concerns the action of $\gamma_2$ on $K_\xi$. \begin{lem}\label{lemma4} Let $\xi \in \mathrm{Par}(K_{(n,n-1)} \otimes  D_{n-1})$ such that $\xi \neq \lambda, \mu$. Then the coefficient of $Z_0 \in K_\xi$ in the expression of  $\pi_{U_0} \circ \gamma_2 (e^\lambda) \in K_\xi$ as a linear combination of semistandard basis elements is nonzero.
\end{lem}
\begin{proof}
	We know that $\xi$ is of the form \[\xi=(n+c_1, n-1+c_2, n-1-c_1-c_2)\] for some nonnegative integers $c_1, c_2, c_3$ satisfying $c_1 +c_2 +c_3 = n-1$ and $c_2 \in \{0,1\}$ according to (\ref{mu}). For notational convenience, let \[c=c_1+c_2.\]
	
	We want to  compute $\pi_{U_0} \circ \gamma_2 (e^\lambda)$ as a linear combination of the basis elements $Z_i$ of $K_\xi$ and in particular we want to determine the coefficient in this linear combination of the basis element $Z_0$.
	
	Using eq. (\ref{phie}) we have
	
	\begin{align}\label{371}
		\pi_{U_0}\circ \gamma_2(e^{\lambda})= &\pi_{U_0}(e^{\lambda}) \\\nonumber&+(-1)^{n}\pi_{U_0}(1^{(n)} \otimes 3^{(n-1)} \otimes 2^{(n-1)})\\\nonumber&  +    (-1)^n\pi_{U_0}(23^{(n-1)} \otimes 12^{(n-2)} \otimes 1^{(n-1)}).
	\end{align}
	Applying Definition \ref{phiS} for \begin{equation*} S=U_0=\begin{matrix*}[l]
			1^{(n)}3^{(c_1)} \\
			2^{(n-1)}3^{(c_2)} \\
			3^{(n-1-c)} \end{matrix*}\end{equation*} we obtain  
		\begin{align}
		&\label{lem3711}\pi_{U_0}(e^{\lambda})=Z_0,\\\label{lem372} &\pi_{U_0}(1^{(n)} \otimes 3^{(n-1)} \otimes 2^{(n-1)})=\pi_\xi(1^{(n)}2^{(c_1)} \otimes 2^{(c_2)}3^{(n-1)} \otimes 2^{(n-1-c)}),\\ &\label{lem373}\pi_{U_0}(23^{(n-1)} \otimes 12^{(n-2)} \otimes 1^{(n-1)})=\tbinom{1+c_2}{1}\pi_\xi(1^{(c_1)}23^{(n-1)} \otimes 1^{(1+c_2)}2^{(n-2)} \otimes 1^{(n-1-c)}).
	\end{align}
	Substituting eqs. (\ref{lem3711}) - (\ref{lem373}) in eq. (\ref{371}) we have 
	
	\begin{align}\label{371*}
		\pi_{U_0}\circ \gamma_2(e^{\lambda})= Z_0 + (-1)^{n}X +(-1)^{n}\tbinom{1+c_2}{1}Y,
	\end{align}
	where 
	\begin{align*} X:&=\pi_\xi(1^{(n)}2^{(c_1)} \otimes 2^{(c_2)}3^{(n-1)} \otimes 2^{(n-1-c)}), \\\
		Y:&=\pi_\xi(1^{(c_1)}23^{(n-1)} \otimes 1^{(1+c_2)}2^{(n-2)} \otimes 1^{(n-1-c)})
		\end{align*}
We want to compute the $X$ and $Y$ as linear combinations of the $Z_i$. To this end we use Lemma \ref{lemglas} repeatedly. 

   Raising the 2's from row 3 of $X$ to row 2 according to Lemma \ref{lemglas}(2), we obtain 
	\[X=(-1)^{n-1-c}\pi_\xi(1^{(n)}2^{(c_1)} \otimes 2^{(n-1-c_1)}3^{(c)} \otimes 3^{(n-1-c)})\] and therefore 
	
		\begin{equation}\label{lem374} X=(-1)^{n-1-c}\pi_\xi(e^{U_{c_1}})=(-1)^{n-1-c}Z_{c_1}.\end{equation}

	For $Y$ we first observe that $c \ge 1$ because if $c=0$, then $c_1=c_2=0$ and thus $\xi=(n,n-1,n-1)=\lambda$, which contradicts the hypothesis $\xi \neq \lambda$ of the Lemma. Therefore $n-2 \ge n-1-c$ and we may apply Lemma \ref{lemglas}(2) to raise the 1's from  row 3 of $Y$ to row 2 to obtain \begin{equation}\label{lem367}Y=(-1)^{n-1-c}Y_{1},\end{equation}
	where \begin{equation}\label{lem367b}Y_1:=\pi_\xi(1^{(c_1)}23^{(n-1)} \otimes 1^{(n-c_1)}2^{(c-1)} \otimes 2^{(n-1-c)}).\end{equation}
	In order to continue the computation of $Y_1$ as a linear combination of the $Z_i$ we distinguish two cases.
	
	\textbf{Case 1}. Suppose $c_1 \ge 1$.
	
	Raising the 1's from row 2 of $Y_1$ to row 1 yields \[Y_1=(-1)^{n-c_1}(Y_2+\tbinom{c}{1}Y_3),\]
	where
	\begin{align*}&Y_2:= \pi_\xi(1^{(n)}23^{(c_1-1)} \otimes 2^{(c-1)}3^{(n-c_1)} \otimes 2^{(n-1-c)}), \\& Y_3:=\pi_\xi(1^{(n)}3^{(c_1)} \otimes 2^{(c)}3^{(n-c_1-1)}\otimes 2^{(n-1-c)}).\end{align*}
	Raising the 2's from row 3 of $Y_2$ to row 2 yields \[Y_2= (-1)^{n-1-c}\pi_\xi(1^{(n)}23^{(c_1-1)} \otimes 2^{(n-2)}3^{(c_2+1)} \otimes 3^{(n-1-c)})=(-1)^{n-1-c}Z_1\]
	and likewise raising the 2's from row 3 of $Y_3$ yields \[Y_3= (-1)^{n-1-c}\pi_\xi(1^{(n)}3^{(c_1)} \otimes 2^{(n-1)}3^{(c_2)} \otimes 3^{(n-1-c)})=(-1)^{n-1-c}Z_0.\]
	By substituting the $Y_i$ in eq. (\ref{lem367}) and we find \[Y= (-1)^{n-c_1}(Z_1+\tbinom{c}{1}Z_0).\]
	Substituting this and eq. (\ref{lem374}) in eq. (\ref{371*}) we find the desired linear combination of semistandard tableaux  
	
	\begin{equation}\label{378}
		\pi_{U_0} \circ \gamma_2 (e^\lambda) = 	Z_0+(-1)^{c+1}Z_{c_1}+(-1)^{c_1}\tbinom{1+c_2}{1}(Z_1 +\tbinom{c}{1}Z_0).
	\end{equation}
	The coefficient of $Z_0$ in the right hand side of eq. (\ref{378}) is equal to $1+(-1)^{c_1}\tbinom{1+c_2}{1} \tbinom{c}{1}.$ This is clearly nonzero if $c_2=1$. If $c_2=0$, then the coefficient is equal to $1+(-1)^{c_1}c_1$. However, for $c_2=0$ we have $c_1 \neq 1$ because $\xi \neq \mu = (n+1,n-1,n-2)$ by hypothesis. Thus we see that the coefficient is nonzero when $c_2=0$. Remembering that $c_2\in \{0,1\}$ we have that the coefficient is nonzero in all cases.
	
	\textbf{Case 2}. Suppose $c_1 =0$. 
	
	In this case we have from eq. (\ref{lem367b}) \[Y_1=\pi_\xi(23^{(n-1)} \otimes 1^{(n)}2^{(c-1)} \otimes 2^{(n-1-c)})\] and we compute similarly to case 1. By raising the 1's from row 2 of $Y_1$ to row 1 we get \[Y_1=(-1)^n\tbinom{c}{1}\pi_\xi(1^{(n)}\otimes2^{(c)}3^{(n-1)}\otimes2^{(n-1-c)})\]
	and by raising in the last term the 2's from row 3 to row 2 we get
	\[Y_1=(-1)^{c+1}\tbinom{c}{1}Z_0.\]
	Substituting this and eq. (\ref{lem374}) in eq. (\ref{371*}) we find the desired linear combination of semistandard tableaux  
	 \begin{equation}\label{379}
		\pi_{U_0} \circ \gamma_2 (e^\lambda)=(1+(-1)^{c_2+1}+\tbinom{1+c_2}{1}\tbinom{c_2}{1})Z_0.
	\end{equation} If $c_2=1$, the coefficient of $Z_0$ in the right hand side of eq. (\ref{379}) is nonzero. We have $c_2 \neq 0$ because $\xi \neq \lambda = (n,n-1,n-1)$ by hypothesis of the Lemma. Remembering that $c_2\in \{0,1\}$ we have that the coefficient is nonzero in all cases.
\end{proof}

\subsection{ Proof of Theorem \ref{main1}}\label{proofmain}	
\begin{proof}	 From Lemma \ref{lemma1} and Lemma \ref{lemma2} we know that the multiplicity of each of $K_{\lambda}$ and $K_{\mu}$ in the cokernel of the map $\gamma_1 +\gamma_2 : D(\lambda) \oplus D(\lambda) \to D(\lambda)$ is equal to 1. 
	
From Corollary \ref{cor38} we know that every irreducible summand of $\coker(\gamma_1+\gamma_2)$ is a summand of $K_{(n,n-1)} \otimes  D_{n-1}$.

 Let us fix $\xi \in \mathrm{Par}(K_{(n,n-1)} \otimes  D_{n-1})$ such that $\xi \neq \lambda$ and $\xi \neq \mu.$ We intend to show that the multiplicity of $K_\xi$ in $D(\lambda)$ is equal to the multiplicity of  $K_\xi$ in $\Ima(\gamma_1)+\Ima(\gamma_2)$, or equivalently, that the vector spaces $\Hom_G(D(\lambda), K_\xi ) $ and $\Hom_G(\Ima(\gamma_1)+\Ima(\gamma_2), K_\xi ) $ have equal dimensions.
	
	We need to recall some notation. \begin{itemize} \item Let 
	$\mathrm{SST_\lambda}(\xi) = \{ U_0,  U_1, \dots, U_q\},$
	where $U_i$ is given by eq. (\ref{Si}). (With the notation of eq. (\ref{Si}) we have $q=c_1$.) \item For each $U_i$ we have the corresponding projection $\pi_{U_i} : D(\lambda) \to K_\xi, e^\lambda \mapsto \pi_\xi(e^{U_i})$, and we know that the $\pi_{U_0}, \dots, \pi_{U_q}$ form a basis of 
	$\Hom_G(D(\lambda), K_{\xi})$. So for the dimension of the vector space $\Hom_G(D(\lambda), K_{\xi})$ we have $\dim \Hom_G(D(\lambda), K_{\xi}) = q+1$. \end{itemize}

	The exact sequence $0 \to \Ima(\gamma_1) \to D(\lambda) \to  \coker(\gamma_1) \to 0$ yields the exact sequence \begin{equation}\label{exact}
		0 \to \Hom_G(\coker(\gamma_1), K_\xi) \to \Hom_G(D(\lambda), K_{\xi}) \to \Hom_G(\Ima(\gamma_1), K_\xi)  \to 0.
	\end{equation}
	From Lemma \ref{lem37} and Pieri's rule, it follows that $\dim\Hom_G(\coker(\gamma_1), K_\xi) =1$ and hence from (\ref{exact}) we conclude that $\dim\Hom_G(\Ima(\gamma_1), K_\xi) =q$. We have that $\Hom_G(\Ima(\gamma_1), K_\xi)$ is generated by the restrictions $\pi_{U_i}\big|_{\Ima{(\gamma_1)}}$, $i=0, \dots, q$, of the maps $\pi_{U_i}$ to $\Ima(\gamma_1).$ Consider the maps \[\pi_{U_0} \circ \gamma_1, \pi_{U_1} \circ \gamma_1 , \dots, \pi_{U_q} \circ \gamma_1 \in \Hom_G(D(\lambda), K_\xi).\] It is clear that if $a_0 \pi_{U_0} \circ \gamma_1 + \cdots + a_q \pi_{U_q} \circ \gamma_1 =0,$
	where $a_i \in \mathbb{K}$, then 
	\[a_0 \pi_{U_0}\big|_{\Ima{(\gamma_1)}} + \cdots + a_q \pi_{U_q}\big|_{\Ima{(\gamma_1)}} =0.\]
	Thus, for the subspace \[W:=\mathrm{span}\{\pi_{U_0} \circ \gamma_1, \pi_{U_1} \circ \gamma_1 , \dots, \pi_{U_q} \circ \gamma_1 \} \] 
	of $\Hom_G(D(\lambda), K_\xi ) $ we have 
	\[ \dim W \ge \dim\Hom_G(\Ima(\gamma_1), K_\xi) =q. \]

	We claim that the subspace \[\mathrm{span}\{\pi_{U_0} \circ \gamma_2, \pi_{U_0} \circ \gamma_1, \pi_{U_1} \circ \gamma_1 , \dots, \pi_{U_q} \circ \gamma_1 \} \] 
	of $\Hom_G(D(\lambda), K_\xi ) $ has dimension  $q+1$. Indeed, to prove this it suffices to show that $ \pi_{U_0} \circ \gamma_2 \notin W. $ Suppose $\pi_{U_0} \circ \gamma_2 = a_0\pi_{U_0} \circ \gamma_1 + \cdots + a_q\pi_{U_q} \circ \gamma_1$, where $a_i \in \mathbb{K}$. Evaluating at $e^\lambda$ we have \[\pi_{U_0} \circ \gamma_2(e^\lambda) = a_0\pi_{U_0} \circ \gamma_1(e^\lambda) + \cdots + a_q\pi_{U_q} \circ \gamma_1(e^\lambda). \] According to Lemma \ref{lemma4}, the coefficient of the semistandard basis element $Z_0=\pi_\xi(e^{U_0})$ in the left  hand side is nonzero, and according to Lemma \ref{lemma3} the coefficient of $Z_0$ in the right hand side is zero. Thus $ \pi_{U_0} \circ \gamma_2 \notin W $ as desired.
	
	It follows from the above claim that there is a basis of $\Hom_G(D(\lambda), K_\xi ) $ consisting of a subset of the elements 
	\begin{equation}\label{maps}\pi_{U_0} \circ \gamma_2, \ \pi_{U_0} \circ \gamma_1, \ \pi_{U_1} \circ \gamma_1 , \ \dots, \ \pi_{U_q} \circ \gamma_1.\end{equation}
	Since every map in (\ref{maps}) factors through a subspace of $\Ima(\gamma_1) + \Ima(\gamma_2)$, we conclude that the multiplicity of $K_\xi$ in $\Ima(\gamma_1) + \Ima(\gamma_2) \subseteq D(\lambda)$ is at least $\dim \Hom_G(D(\lambda), K_\xi )$. Thus this multiplicity is equal to $\dim \Hom_G(D(\lambda), K_\xi )$. Consequently, the multiplicity of $K_\mu$ in $\coker(\gamma_1 + \gamma_2)$ is equal to zero.
\end{proof}

\subsection{The map $\gamma_3$}\label{sec35} 

\begin{defn}\label{mapg3} Define $\gamma_3  \in \Hom _G(D(\nu), D(\lambda))$ by \[\gamma_3 := \phi_{Q(1)}+ (-1)^n\phi_{Q(2)},\] where the tableaux $Q(i) \in \mathrm{RSST}_\nu(\lambda) $ are the following
	\[ Q(1):=\begin{matrix*}[l]
		1^{(n)}  \\
		23^{(n-2)} \\
		2^{(n-1)} \end{matrix*}, \; Q(2):=\begin{matrix*}[l]
		2^{(n)}  \\
		13^{(n-2)} \\
		1^{(n-1)} \end{matrix*}. \]	
\end{defn}
\begin{lem}We have \begin{align}\label{gamma3e}&\gamma_3(e^{\nu})= 1^{(n)} \otimes 23^{(n-2)} \otimes 2^{(n-1)} +(-1)^n 2^{(n)} \otimes 13^{(n-2)}\otimes1^{(n-1)}.\end{align} \end{lem}
\begin{proof} This is clear from the previous definition and Definition \ref{phiS}(1).
\end{proof}

The reason for considering the map $\gamma_3$ will become apparent in Section 5.1. Roughly speaking, the image of $\gamma_3$ corresponds to a particular relation of $\text{Lie}_n(m)$, namely to relation (R5) of Lemma \ref{R5}.

At the beginning of the proof of Lemma \ref{lemma2}, we observed the following.
\begin{rem}\label{twoss}
	There are exactly two semistandard tableaux of shape $\mu$ and weight $\lambda$, 
	
	\[ R_1:=\begin{matrix*}[l]
		1^{(n)}3  \\
		2^{(n-1)} \\
		3^{(n-2)} \end{matrix*}, \; R_2:=\begin{matrix*}[l]
		1^{(n)}2  \\
		2^{(n-2)}3 \\
		3^{(n-2)} \end{matrix*}.  \]
	According to Section 2.3, the corresponding projections are 	\begin{align*} &\pi_{R_1}:D(\lambda) \to K_\mu, e^{\lambda} \mapsto \pi_\mu(1^{(n)}3\otimes2^{(n-1)} \otimes 3^{(n-2)}), \\&\pi_{R_2}:D(\lambda) \to K_\mu, e^{\lambda} \mapsto \pi_\mu(1^{(n)}2\otimes2^{(n-2)}3 \otimes 3^{(n-2)}).\end{align*}
\end{rem}

\begin{lem}\label{gamma3}With the above notation, each of the following compositions of $G$-maps is the zero map
\begin{enumerate}
	\item $D(\nu) \xrightarrow{\gamma_3} D(\lambda) \xrightarrow{\pi_\lambda} K_\lambda$,
	\item $D(\nu) \xrightarrow{\gamma_3} D(\lambda) \xrightarrow{\pi_{R_1}} K_\mu$,
	\item $D(\nu) \xrightarrow{\gamma_3} D(\lambda) \xrightarrow{\pi_{R_2}} K_\mu$.
\end{enumerate}
\end{lem}
\begin{proof}
(1) Using eq. (\ref{gamma3e}) we have \[\pi_\lambda \circ \gamma_3 (e^\nu)=\pi_\lambda\big(1^{(n)} \otimes 23^{(n-2)} \otimes 2^{(n-1)}\big) +(-1)^n \pi_\lambda\big(2^{(n)} \otimes 13^{(n-2)}\otimes 1^{(n-1)}\big).\]
The number of 2's in rows 2 and 3 of $\pi_\lambda\big(1^{(n)} \otimes 23^{(n-2)} \otimes 2^{(n-1)}\big)$ is equal to $1+n-1=n$ and the length of the second part of the partition $\lambda$ is equal to $n-1$. Since $n-1 < n$, we conclude from Lemma \ref{lemglas}(1) that \[\pi_\lambda\big(1^{(n)} \otimes 23^{(n-2)} \otimes 2^{(n-1)}\big)=0.\] In a similar manner we also have  \[\pi_\lambda\big(2^{(n)} \otimes 13^{(n-2)}\otimes 1^{(n-1)}\big)=0.\] Hence $\pi_\lambda \circ \gamma_3 (e^\nu)=0.$ Since $e^{\nu}$ generates $D(\nu)$ as a $G$-module (Remark \ref{cyclic}) and $\pi_\lambda \circ \gamma_3$ is a map of $G$-modules, we conclude that $\pi_\lambda \circ \gamma_3=0.$

(2) Using eq. (\ref{gamma3e}) we have \begin{equation}\label{newlemma1}\pi_{R_1} \circ \gamma_3 (e^\nu)=\pi_{R_1}\big(1^{(n)} \otimes 23^{(n-2)} \otimes 2^{(n-1)}\big) +(-1)^n \pi_{R_1}\big(2^{(n)} \otimes 13^{(n-2)}\otimes1^{(n-1)}\big).\end{equation} We have $\pi_{R_1} = \pi_\mu \circ \phi_{R_1}$ according to Definition \ref{phiS}(2). 

We compute the first summand in the right hand side of eq. (\ref{newlemma1}). Using the definition of $\phi_{R_1} $ (i.e. Definition \ref{phiS}(1)),  we obtain 
\[
\pi_{R_1} \big(1^{(n)} \otimes 23^{(n-2)} \otimes 2^{(n-1)}\big)=\pi_{\mu} \big(1^{(n)}2 \otimes 23^{(n-2)} \otimes 2^{(n-2)}\big).
\]
Applying Lemma \ref{lemglas}(2) to the right hand side of the above equation for rows 2 and 3, we obtain \[ \pi_{\mu} \big(1^{(n)}2 \otimes 23^{(n-2)} \otimes 2^{(n-2)}\big)=  (-1)^{(n-2)}\pi_{\mu} \big(1^{(n)}2 \otimes 2^{(n-1)} \otimes 3^{(n-2)}\big).\] Hence 
\begin{equation}\label{newlemma2}
	\pi_{R_1} \big(1^{(n)} \otimes 23^{(n-2)} \otimes 2^{(n-1)}\big)=(-1)^{(n-2)}\pi_{\mu} \big(1^{(n)}2 \otimes 2^{(n-1)} \otimes 3^{(n-2)}\big)
\end{equation}

We compute now the second summand in the right hand side of eq. (\ref{newlemma1}). The definition of $\phi_{R_1}$ yields \[ \pi_{R_1}\big(2^{(n)} \otimes 13^{(n-2)}\otimes1^{(n-1)}\big)=\pi_{\mu}\big(12^{(n)} \otimes 13^{(n-2)}\otimes1^{(n-2)}\big).\]
Applying Lemma \ref{lemglas}(2) to the right hand side of the above equation for rows 2 and 3, we obtain
\[ \pi_{R_1}\big(2^{(n)} \otimes 13^{(n-2)}\otimes1^{(n-1)}\big)=(-1)^{n-2}\pi_{\mu}\big(12^{(n)} \otimes 1^{(n-1)}\otimes3^{(n-2)}\big).\]
Applying the same lemma to the right hand side of the above equation for rows 1 and 2 yields \[(-1)^{n-2}\pi_{\mu}\big(12^{(n)} \otimes 1^{(n-1)}\otimes3^{(n-2)}\big)=(-1)^{n-2}(-1)^{n-1}\pi_{\mu}\big(1^{(n)}2 \otimes 2^{(n-1)}\otimes3^{(n-2)}\big).\] Hence
\begin{equation}\label{newlemma3}
	\pi_{R_1} \big(2^{(n)} \otimes 13^{(n-2)} \otimes 1^{(n-1)}\big)=-\pi_{\mu} \big(1^{(n)}2 \otimes 2^{(n-1)} \otimes 3^{(n-2)}\big).
\end{equation}

Now we substitute (\ref{newlemma2}) and (\ref{newlemma3}) into (\ref{newlemma1}) to obtain
\begin{align*}\pi_{R_1}& \circ \gamma_3 (e^\nu)\\&=(-1)^{(n-2)}\pi_{\mu} \big(1^{(n)}2 \otimes 2^{(n-1)} \otimes 3^{(n-2)}\big) - (-1)^n \pi_{\mu}\big(1^{(n)}2 \otimes 2^{(n-1)} \otimes 3^{(n-2)}\big)\\&=0.\end{align*}
Hence $\pi_{R_1} \circ \gamma_3=0.$

(3) This computation is similar to (2) but has two extra steps. Using eq. (\ref{gamma3e}) we have \begin{equation}\label{newlemma4}\pi_{R_2} \circ \gamma_3 (e^\nu)=\pi_{R_2}\big(1^{(n)} \otimes 23^{(n-2)} \otimes 2^{(n-1)}\big) +(-1)^n \pi_{R_2}\big(2^{(n)} \otimes 13^{(n-2)}\otimes1^{(n-1)}\big).\end{equation}

We compute the first summand in the right hand side of eq. (\ref{newlemma4}). Using the definition of $\phi_{R_2} $, see (\ref{phiS}),  we obtain 
\begin{align}\label{newlemma5}
	\pi_{R_2}\big(1^{(n)} \otimes 23^{(n-2)} \otimes 2^{(n-1)}\big)= &\pi_{\mu}\big(1^{(n)}2 \otimes 23^{(n-2)} \otimes 2^{(n-2)}\big) \\\nonumber&+\tbinom{2}{1}\pi_{\mu}\big(1^{(n)}3 \otimes 2^{(2)}3^{(n-3)} \otimes 2^{(n-2)}\big).
\end{align}
Applying Lemma \ref{lemglas}(2) to $\pi_{\mu}\big(1^{(n)}2 \otimes 23^{(n-2)} \otimes 2^{(n-2)}\big)$ for rows 2 and 3, we have \[ \pi_{\mu}\big(1^{(n)}2 \otimes 23^{(n-2)} \otimes 2^{(n-2)}\big) = (-1)^{n-2}\pi_{\mu}\big(1^{(n)}2 \otimes 2^{(n-1)} \otimes 3^{(n-2)}\big).\]
The number of 2's in rows 2 and 3 of $\pi_\mu\big(1^{(n)}3 \otimes 2^{(2)}3^{(n-3)} \otimes 2^{(n-2)}\big)$ is equal to $2+n-2=n$ and the length of the second part of the partition $\mu$ is equal to $n-1$. Since $n-1 < n$, we conclude from Lemma \ref{lemglas}(1) that \[\pi_\mu\big(1^{(n)}3 \otimes 2^{(2)}3^{(n-3)} \otimes 2^{(n-2)}\big)=0.\] Substituting in (\ref{newlemma5}) we get \begin{equation}\label{newlemma55}
	\pi_{R_2}\big(1^{(n)} \otimes 23^{(n-2)} \otimes 2^{(n-1)}\big)=(-1)^{n-2}\pi_{\mu}\big(1^{(n)}2 \otimes 2^{(n-1)} \otimes 3^{(n-2)}\big).
\end{equation}

Now we compute the second summand in the right hand side of eq. (\ref{newlemma4}). The definition of $\phi_{R_2}$ yields 
\begin{align}\label{newlemma6}
	\pi_{R_2}\big(2^{(n)} \otimes 13^{(n-2)} \otimes 1^{(n-1)}\big)= &\pi_{\mu}\big(12^{(n)} \otimes 13^{(n-2)} \otimes 1^{(n-2)}\big) \\\nonumber&+\tbinom{2}{1}\pi_{\mu}\big(2^{(n)}3 \otimes 1^{(2)}3^{(n-3)} \otimes 1^{(n-2)}\big).
\end{align}
Applying Lemma \ref{lemglas}(2) to $\pi_{\mu}\big(12^{(n)} \otimes 13^{(n-2)} \otimes 1^{(n-2)}\big)$ for rows 2 and 3, we have \[ \pi_{\mu}\big(12^{(n)} \otimes 13^{(n-2)} \otimes 1^{(n-2)}\big)=(-1)^{n-2}  \pi_{\mu}\big(12^{(n)} \otimes 1^{(n-1)} \otimes 3^{(n-2)}\big)\] 
and applying the same lemma to the resulting term $\pi_{\mu}\big(12^{(n)} \otimes 1^{(n-1)} \otimes 3^{(n-2)}\big)$ for rows 1 and 2, we have \[ \pi_{\mu}\big(12^{(n)} \otimes 1^{(n-1)} \otimes 3^{(n-2)}\big) =(-1)^{n-1}\pi_{\mu}\big(1^{(n)}2 \otimes 2^{(n-1)} \otimes 3^{(n-2)}\big).\] Hence substituting we obtain \begin{align}\label{newlemma7}
\pi_{\mu}\big(12^{(n)} \otimes 13^{(n-2)} \otimes 1^{(n-2)}\big)&=(-1)^{n-2}(-1)^{n-1}\pi_{\mu}\big(1^{(n)}2 \otimes 2^{(n-1)} \otimes 3^{(n-2)}\big)\\\nonumber&= - \pi_{\mu}\big(1^{(n)}2 \otimes 2^{(n-1)} \otimes 3^{(n-2)}\big).
\end{align}
The number of 1's in rows 2 and 3 of $\pi_\mu\big(2^{(n)}3 \otimes 1^{(2)}3^{(n-3)} \otimes 1^{(n-2)}\big)$ is equal to $2+n-2=n$ and the length of the second part of the partition $\mu$ is equal to $n-1$. Since $n-1 < n$, we conclude from Lemma \ref{lemglas}(1) that \begin{equation}\label{newlemma8}\pi_\mu\big(2^{(n)}3 \otimes 1^{(2)}3^{(n-3)} \otimes 1^{(n-2)}\big)=0.\end{equation}

Now substituting (\ref{newlemma7}) and (\ref{newlemma8}) into (\ref{newlemma6}) we get 

\begin{equation}\label{newlemma9}
	\pi_{R_2}\big(2^{(n)} \otimes 13^{(n-2)} \otimes 1^{(n-1)}\big)= -\pi_{\mu}\big(1^{(n)}2 \otimes 2^{(n-1)} \otimes 3^{(n-2)}\big)
\end{equation}
and substituting (\ref{newlemma55}) and (\ref{newlemma9}) into (\ref{newlemma4}) we get
\begin{align*}\pi_{R_2} \circ \gamma_3 (e^\nu)&=(-1)^{n-2}\pi_{\mu}\big(1^{(n)}2 \otimes 2^{(n-1)} \otimes 3^{(n-2)}\big) \\& \ \ \ -(-1)^n \pi_{\mu}\big(1^{(n)}2 \otimes 2^{(n-1)} \otimes 3^{(n-2)}\big)\\& =0.\end{align*}
Hence $\pi_{R_2} \circ \gamma_3=0.$
\end{proof}
\subsection{Main result for Weyl modules} Recall that we have the partitions $\lambda, \mu$ and $\nu $ of $3n-2$ defined at the beginning of Section \ref{sec3}. Also we have the maps of $G$-modules \begin{align*}\gamma_i&: D(\lambda) \to D(\lambda) \ (i=1,2), \\ \gamma_3&: D(\nu) \to D(\lambda)\end{align*}
given in Definition \ref{mapg} and Definition \ref{mapg3} respectively. We defined the map of $G$-modules $\gamma_1 + \gamma_2 + \gamma_3 : D(\lambda) \oplus D(\lambda) \oplus D(\nu) \to D(\lambda)$ by $(\gamma_1 + \gamma_2 + \gamma_3) (x,y,z) = \gamma_1(x) + \gamma_2(y) + \gamma_3(z)$, where $x,y \in D(\lambda)$  and $z \in D(\nu)$. It follows that $\Ima (\gamma_1 + \gamma_2 + \gamma_3)= \Ima(\gamma_1 + \gamma_2)+ \Ima(\gamma_3)=$$\Ima(\gamma_1) + \Ima(\gamma_2)+ \Ima(\gamma_3)$. The main result of this paper for Weyl modules is the following.
\begin{thm}\label{main2}
	Let $N \ge 3n-2$. Then, the cokernel of the map \[\gamma_1 + \gamma_2 + \gamma_3 : D(\lambda) \oplus D(\lambda) \oplus D(\nu) \to D(\lambda) \] is isomorphic to $K_{\lambda} \oplus K_{\mu}$ as $G$-modules.
\end{thm}
\begin{proof} From Theorem \ref{main1} we know that the cokernel of the map \[ \gamma_1 + \gamma_2: D(\lambda) \oplus D(\lambda) \to D(\lambda) \]
	is isomorphic as a $G$-module to $K_\lambda \oplus K_\mu$. Hence it suffices to show that each of the irreducibles $K_{\lambda}$ and $K_{\mu}$ has multiplicity equal to zero in the image $\Ima (\gamma_3)$ of the map \[\gamma_3 : D(\nu) \to D(\lambda). \]
	
	We know that a basis of $\Hom_G(D(\lambda), K_\lambda)$ is the set $\{ \pi_\lambda\}.$ Thus from Lemma \ref{gamma3}(1) we conclude that the multiplicity of $K_\lambda$ in $\Ima(\gamma_3)$ is equal to zero.
	
	We know from Remark \ref{twoss} and Proposition \ref{wbasis} that a basis of $\Hom_G(D(\lambda), K_\mu)$ is the set \[\{ \pi_{R_1}, \pi_{R_2}\}.\]
	Thus from Lemma \ref{gamma3}(2), (3) we conclude that the multiplicity of $K_\mu$ in $\Ima(\gamma_3)$ is equal to zero.
\end{proof}
\begin{rem}\label{rem46}According to Theorem \ref{main1} and Theorem \ref{main2}, the cokernels of the maps $\gamma_1 + \gamma_2 : D(\lambda) \oplus D(\lambda) \to D(\lambda)$ and $\gamma_1 + \gamma_2 + \gamma_3 : D(\lambda) \oplus D(\lambda) \oplus D(\nu) \to D(\lambda)$ are isomorphic $G$-modules. This implies that $\Ima(\gamma_3)$ is contained in $\Ima(\gamma_1+\gamma_2)$ because $\Ima(\gamma_1+\gamma_2) \subseteq \Ima(\gamma_1+\gamma_2+\gamma_3)$.
	\end{rem}

\section{LAnKes and Specht modules}
Let $m=3n-2$ and $n \ge 2$.  The purpose of this section is to show how Theorem \ref{mainspecht} follows from Theorem \ref{main2}. We will need a particular presentation of $\text{Lie}_n(m)$ given in Lemma \ref{presentationlanke} below.
\subsection{A presentation of $\text{Lie}_n(m)$} 

First, we will describe a presentation of $\text{Lie}_n(m)$ (see Lemma \ref{genrel1} below) in the spirit of Section 2.2 of \cite{FHSW}. We recall that an $n$-bracketed permutation on $[m]$ is an $n$-bracketed word on $[m]$ such that each $a \in [m]$ appears exactly once. The symmetric group $\mathfrak{S}_m$ acts naturally on $n$-bracketed permutations by replacing each $i$ of the bracketed permutation by $\sigma(i)$.

\begin{defn} Let $W=W_{n,3}$ be the vector space generated by all possible $n$-bracketed permutations on $[m]$ subject only to skew commutativity of the bracket given in Definition  \ref{lanke}(1) (but not to the generalized Jacobi identity  (\ref{GJI})).
\end{defn}
It is clear that $W$ is an $\mathfrak{S}_m$-module.

Here the number of brackets is $k=3$. Hence up to skew commutativity there are two types of generators of $W$ consisting of bracketed permutations
\begin{align}\tag{G1} &[[[{x_1},\dots ,{x_n} ],{y_{1}},\dots, y_{n-1}], z_{1},\dots ,z_{n-1}  ], \\&\tag{G2}[[{x_1},\dots ,{x_n} ],[{y_{1}},\dots, y_{n}], z_{1},\dots ,z_{n-2}  ],
\end{align}
for all $x_i, y_j, z_k \in [m]$.
\begin{lem}\label{Wgen} The subspace of $W$ consisting of the relations satisfied by the elements \textup{(G1)} and \textup{(G2)} is generated by the following relations
	\begin{align}
		\label{51}\textup{(G1)}&- \sign(\sigma)[[[{x_{\sigma(1)}},\dots ,{x_{\sigma(n)}} ],{y_{1}},\dots, y_{n-1}], z_{1},\dots ,z_{n-1}], \ \sigma \in \mathfrak{S}_n, \\\label{52}\textup{(G1)}&-\sign(\tau)[[[{x_{1}},\dots ,{x_{n}} ],{y_{\tau(1)}},\dots, y_{\tau(n-1)}], z_{1},\dots ,z_{n-1}], \ \tau \in \mathfrak{S}_{n-1},\\\label{53}\textup{(G1)}&-\sign(\tau)[[[{x_{1}},\dots ,{x_{n}} ],{y_{1}},\dots, y_{n-1}], z_{\tau(1)},\dots ,z_{\tau(n-1)}], \ \tau \in \mathfrak{S}_{n-1},
	\end{align}
	and
	\begin{align}
		\label{54}\textup{(G2)}&-\sign(\sigma)[[{x_{\sigma(1)}},\dots ,{x_{\sigma(n)}} ],[{y_{1}},\dots, y_{n}],  z_{1},\dots ,z_{n-2}  ],  \ \sigma \in \mathfrak{S}_n,
		\\\label{55}\textup{(G2)}&-\sign(\sigma)[[{x_1},\dots ,{x_n} ],[{y_{\sigma(1)}},\dots, y_{\sigma(n)}], z_{1},\dots ,z_{n-2}  ], \ \sigma \in \mathfrak{S}_n,\\
		\label{56}\textup{(G2)}&-\sign(\tau)[[{x_1},\dots ,{x_n} ],[{y_{1}},\dots, y_{n}], z_{\tau(1)},\dots ,z_{\tau(n-2)}  ], \ \tau \in \mathfrak{S}_{n-2},
		\\\label{57}\textup{(G2)}&+ [[{y_1},\dots ,{y_n} ],[{x_{1}},\dots, x_{n}], z_{1},\dots ,z_{n-2} ].
	\end{align}
	\end{lem}
\begin{proof}
	This follows immediately from the definition of $W$ since all relations in $W$ are consequences of the identities $[x_1, x_2, \dots, x_n]=\sign(\sigma)[x_{\sigma(1)}, x_{\sigma(2)}, \dots, x_{\sigma(n)}$] , where $x_1, \dots, x_n \in [m]$ are distinct elements and  $\sigma \in \mathfrak{S}_n$.
\end{proof}

Remarks on notation: (1) According to Definition \ref{lanke}, the vector space $\text{Lie}_n(m)$ is a quotient of $W$ (by the subspace generated by relations given by the generalized Jacobi identity). We will use the same symbol for a bracketed permutation in $W$ and the corresponding coset in $\text{Lie}_n(m)$ if there is no danger of confusion. Thus, when considering elements of $\text{Lie}_n(m)$, we may refer to (G1) and (G2) as generators of $\text{Lie}_n(m)$. (2) For a sequence $x_1, \dots, x_q$ of elements of $[m]$ and for $i \in \{1,\dots, q\}$, we denote by $x_1, \dots, \widehat{x_i}, \dots, x_q$ the sequence obtained by omitting the term $x_i$.

\begin{lem}\label{genrel1}
	Let $\spn (R1,R2,R3)$ be the subspace of $W$ spanned by the elements $\textup{(R1), (R2), (R3)}$  defined as follows   \begin{align}\tag{R1} \textup{(G1)}&- \sum_{i=1}^{n}(-1)^{i-1}[[[x_i,{y_{1}},\dots, y_{n-1}],{x_1},\dots ,\widehat{x_i}, \dots, x_n ],z_{1},\dots ,z_{n-1}],\\
		\tag{R2} \textup{(G1)}&- [[[{x_1},\dots ,{x_n} ],{z_{1}},\dots, z_{n-1}], y_{1},\dots ,y_{n-1}]
		\\\nonumber&-\sum_{i=1}^{n-1}(-1)^{i-1}[[{x_1},\dots ,{x_{n}}],[y_{i},z_1, \dots, z_{n-1} ],{y_{1}},\dots, \widehat{y_i},\dots ,y_{n-1}],\\\tag{R3} \textup{(G2)}&-\sum_{i=1}^{n}(-1)^{i}[[[y_1,\dots,y_n],x_i , z_1,\dots, z_{n-2}],x_1,\dots,\widehat{x_i},\dots, x_n],
	\end{align}
for all $x_i, y_j, z_k \in [m]$. Then the $\mathfrak{S_m}$-modules  $\text{Lie}_n(m)$ and $W/\spn (R1,R2,R3)$ are isomorphic. 
\end{lem}
\begin{proof} The relations among the generators (G1) and (G2) of $\text{Lie}_n(m)$ are consequences of skew commutativity of the bracket and the generalized Jacobi identity, cf. Definition \ref{lanke}. Up to skew commutativity of the bracket, this identity may be applied in exactly two ways to a generator of type (G1); we may exchange the $y_1, \dots, y_{n-1}$ with $n-1$ of the $x_1, \dots, x_n$ (keeping the $z_k$ fixed), or we may exchange the $z_1, \dots, z_{n-1}$ with $n-1$ of the $[x_1, \dots, x_n], y_1, \dots, y_{n-1}$. In the first case, we obtain \[ \textup{(G1)} = \sum_{i=1}^{n}[[{x_1},\dots ,{x_{i-1}}, [x_i,{y_{1}},\dots, y_{n-1}], x_{i+1}, \dots, x_n ],z_{1},\dots ,z_{n-1}]. \] From this and skew commutativity of the bracket (we move the ${x_1},\dots ,{x_{i-1}}$ to the right of the element $[x_i,{y_{1}},\dots, y_{n-1}]$), we obtain \textup{(R1)}. In the second case, we obtain \textup{(R2)} in a similar manner.
	
	Up to skew commutativity of the bracket, the generalized Jacobi identity may be applied in exactly one way to a generator of type (G2), by exchanging the $[y_1, \dots, y_n], z_1, \dots, z_{n-2}$ with $n-1$ of the $x_1, \dots, x_n$. Thus we obtain (R3).
\end{proof}

In the next lemma we prove two more relations of $\text{Lie}_n(m)$.

\begin{lem}\label{R5} Define the elements of $W$
	\begin{align}\tag{R4}\textup{(G1)}&-[[[{x_1},\dots ,{x_n} ],{z_{1}},\dots, z_{n-1}], y_{1},\dots ,y_{n-1}]\\\nonumber&-\sum_{i=1}^{n-1}(-1)^{i-1}\sum_{j=1}^n(-1)^j[[[y_{i},z_1, \dots, z_{n-1}],x_j,y_1,\dots,\widehat{y}_{i}, \dots, y_{n-1}],\\&\nonumber x_1, \dots, \widehat{x}_j, \dots, x_n],\\\tag{R5}
		&\sum_{i=1}^{n}(-1)^{i}[[[x_1,\dots,x_n],y_i, z_1,\dots,z_{n-2}],y_1,\dots, \widehat{y_{i}}, \dots, y_n] \\\nonumber&+ \sum_{i=1}^{n}(-1)^{i}[[[y_1,\dots,y_n],x_i, z_1,\dots,z_{n-2}],x_1,\dots, \widehat{x_{i}}, \dots, x_n],
	\end{align}
	for all $x_i, y_j, z_k \in [m]$. The images of these in $\text{Lie}_n(m)$ are equal to 0.
\end{lem}
\begin{proof}
	Note that  by skew commutativity of the bracket, every summand in the sum of (R2) of Lemma \ref{genrel1} is up to sign  a generator of type (G2). We substitute (R3) in (R2) to obtain (R4).
	
	We have the element (R3) of Lemma \ref{genrel1} and the corresponding element obtained by exchanging the $x$'s and $y$'s. By adding these and using skew commutativity of the bracket we obtain (R5).\end{proof}

\begin{defn} Let $W(1)$ be the $\mathfrak{S}_m$-submodule of $W$ generated by the elements (G1).
\end{defn}

We note that the elements (R1), (R4) and (R5) of $W$ involve only generators of type (G1). Hence the subspace $\spn(R1,R4,R5)$ spanned by these is contained in $W(1)$. Also it is clear that $\spn(R1,R4,R5)$ is an $\mathfrak{S}_m$-submodule of $W(1)$.

\begin{defn}\label{liebar}
	Let $\overline{\text{Lie}}_n(m)$ be the $\mathfrak{S}_m$-module $W(1)/\spn(R1, R4, R5)$.\end{defn}
\begin{lem}\label{liebarsur} The inclusion map $W(1) \subseteq W$ induces a surjective map of $\mathfrak{S}_m$-modules \[\overline{\textup{Lie}}_n(m) \to \textup{Lie}_n(m).\]
\end{lem}
\begin{proof}
From Lemma \ref{genrel1} it follows that the image of (R1) in $\text{Lie}_n(m)$  is equal to 0 and from Lemma \ref{R5} it follows that the images of (R4) and (R5) in $\text{Lie}_n(m)$ are equal to 0. Hence the inclusion map $W(1) \subseteq W$ induces a map $W(1)/\spn(R1, R4, R5) \to \text{Lie}_n(m)$.

We know that $\text{Lie}_n(m)$ is generated by the images of (G1) and (G2). Hence by Lemma \ref{genrel1}(R3), $\text{Lie}_n(m)$ is generated by the images of (G1). So the map $W(1)/\spn(R1, R4, R5) \to \text{Lie}_n(m)$ is surjective.
\end{proof}

We will show now that the surjective map $\overline{\text{Lie}}_n(m) \to \text{Lie}_n(m)$ of the previous lemma is an isomorphism. For this we need a particular relation in $\overline{\text{Lie}}_n(m)$ (Lemma \ref{finalrel} below). For the proof we will make use of a remark on the  sign of permutations that follows.
\begin{rem}\label{sign} Recall that for a finite sequence $u_1, \dots, u_k $ of distinct positive integers, the inversion number $inv(u_1,...,u_k)$ is the number of pairs $(u_i, u_j)$ such that $i<j$ and $u_i >u_j$. For a permutation $\sigma \in \mathfrak{S}_n$ we have $\sign(\sigma)= (-1)^{inv(\sigma(1), \dots,  \sigma(n))}$. Now let $i \in \{1,...,n\}$. We may consider the inversion number $inv(\sigma(1), \dots, \widehat{\sigma(i)}, \dots, \sigma(n))$ of the sequence obtained from $\sigma(1),  \dots, \sigma(n)$ by deleting the term $\sigma(i)$. We have \begin{equation}\label{inv}
(-1)^i(-1)^{inv(\sigma(1), \dots, \widehat{\sigma(i)}, \dots, \sigma(n))}=(-1)^{\sigma(i)}\sign(\sigma).
	\end{equation}
\begin{proof} Working modulo 2 we have
	\begin{align*}
\sign(\sigma) &\equiv inv(\sigma(1), \dots, \sigma(n))\\&\equiv inv(\sigma(i), \sigma(1), \dots, \widehat{\sigma(i)}, \dots, \sigma(n)) +(i-1)
\\&\equiv inv(\sigma(i), 1, \dots, \widehat{\sigma(i)}, \dots, n) + inv(\sigma(1), \dots, \widehat{\sigma(i)}, \dots, \sigma(n))+(i-1)\\&\equiv inv(1, \dots, n) + (\sigma(i)-1) + inv(\sigma(1), \dots, \widehat{\sigma(i)}, \dots, \sigma(n))+(i-1)	
\\&\equiv \sigma(i) + inv(\sigma(1), \dots, \widehat{\sigma(i)}, \dots, \sigma(n))+i.\end{align*} In the second and fourth congruences we used the fact $inv(u) \equiv inv(u')+1$ if the sequence $u'$ is obtained from $u$ by transposing two adjacent elements. In the third we used that $inv(u_1,u_2,\dots,u_k) \equiv inv(u_1,v_2,\dots,v_k) + inv(u_2,\dots,u_k)$, where $v_2 < \dots < v_k$ is the increasing sequence obtained from $u_2, \dots, u_k$ by rearranging the terms in increasing order.
\end{proof}
	\end{rem}

\begin{lem}\label{finalrel}
In $W(1)$  we have the relations 
\begin{align*}&\sum_{i=1}^n(-1)^i[[[y_1,\dots,y_n],x_i,z_1,\dots, z_{n-2}],x_1,\dots, \widehat{x_i}, \dots, x_n]\\&-\sign(\sigma)\sum_{i=1}^n(-1)^i[[[y_1,\dots,y_n],x_{\sigma(i)},z_1,\dots, z_{n-2}],x_{\sigma(1)},\dots, \widehat{x_{\sigma(i)}}, \dots, x_{\sigma(n)}],
\end{align*} where $\sigma \in \mathfrak{S}_n$. Thus we have the corresponding relations in $\overline{\textup{Lie}}_n(m)$.
\end{lem}
\begin{proof}
	For  $i \in \{1,\dots,n\}$ we know from relation (\ref{53}) that \[[[[y_1,\dots,y_n],x_i,z_1,\dots, z_{n-2}],x_1,\dots, \widehat{x_i}, \dots, x_n]\] is skew commutative in the $x_1, \dots ,x_{i-1}, x_{i+1}, x_n.$ Hence for every $\sigma \in \mathfrak{S}_n$ we have  \begin{align*}
	&\sum_{i=1}^n(-1)^i[[[y_1,\dots,y_n],x_{\sigma(i)},z_1,\dots, z_{n-2}],x_{\sigma(1)},\dots, \widehat{x_{\sigma(i)}}, \dots, x_{\sigma(n)}]\\&=\sum_{i=1}^n(-1)^i(-1)^{{inv(\sigma(1), \dots, \widehat{\sigma(i)}, \dots, \sigma(n))}}[[[y_1,\dots,y_n],x_{\sigma(i)},z_1,\dots, z_{n-2}],x_{1},\dots, \widehat{x_{\sigma(i)}}, \dots, x_{n}]\\&\overset{(\ref{inv})}{=}
	\sum_{i=1}^n(-1)^{\sigma (i)}\sign(\sigma)[[[y_1,\dots,y_n],x_{\sigma(i)},z_1,\dots, z_{n-2}],x_{1},\dots, \widehat{x_{\sigma(i)}}, \dots, x_{n}].
		\end{align*} 

The right hand side is equal to 
\begin{align*}
	&\sign(\sigma)\sum_{i=1}^n(-1)^{\sigma (i)}[[[y_1,\dots,y_n],x_{\sigma(i)},z_1,\dots, z_{n-2}],x_{1},\dots, \widehat{x_{\sigma(i)}}, \dots, x_{n}]\\&=\sign(\sigma)\sum_{i=1}^n(-1)^{i}[[[y_1,\dots,y_n],x_{i},z_1,\dots, z_{n-2}],x_{1},\dots, \widehat{x_{i}}, \dots, x_{n}],
\end{align*}
where in the last equality we have a rearrangement of the terms in the sum. The first result of the lemma follows. The second follows from the first since $\overline{\textup{Lie}}_n(m)$ is a quotient of W(1). \end{proof} The main result of the present subsection is the following.
\begin{lem}\label{presentationlanke} The inclusion map $W(1) \subseteq W$ induces an isomorphism of $\mathfrak{S}_m$-modules \[\overline{\textup{Lie}}_n(m) \to \textup{Lie}_n(m).\]
\end{lem}
\begin{proof}  From Lemma \ref{liebarsur} we know that the inclusion map $W(1) \subseteq W$ induces a surjective linear map  $ \overline{\text{Lie}}_n(m) \to \text{Lie}_n(m).$ 
	
	In the argument given below, we show that there exists a surjective linear map $\text{Lie}_n(m) \to \overline{\text{Lie}}_n(m)$. Since these spaces are finite dimensional, the surjective linear map $\overline{\text{Lie}}_n(m) \to \text{Lie}_n(m)$  is an isomorphism as desired.
	
	Consider the map \[ \Theta:W \to \overline{\text{Lie}}_n(m)\]  \begin{itemize}
		\item that is the identity map on (G1), and
		\item sends a generator $[[{x_1},\dots ,{x_n} ],[{y_{1}},\dots, y_{n}], z_{1},\dots ,z_{n-2}  ]$ of type (G2) to 
		\begin{equation}\label{ThetaG2}\sum_{i=1}^{n}(-1)^{i}[[[y_1,\dots,y_n],x_i , z_1,\dots, z_{n-2}],x_1,\dots,\widehat{x_{i}},\dots, x_n].\end{equation}
	\end{itemize}
	We need to show that $\Theta$ is well defined, i.e. sends the relations  of  $W$ to relations of $\overline{\text{Lie}}_n(m)$ .
	
	Since $\Theta$ is the identity map on (G1), it is clear that the relations (\ref{51}) - (\ref{53})  are sent to the corresponding relations of $\overline{\text{Lie}}_n(m)$.
	
	Next we consider the relations (\ref{54}) - (\ref{57}) of $W$. From the definition of $\Theta$, it follows that the image of the relation (\ref{54}) is equal to 
	\begin{align*}&\sum_{i=1}^n(-1)^i[[[y_1,\dots,y_n],x_i,z_1,\dots, z_{n-2}],x_1,\dots, \widehat{x_i}, \dots, x_n]\\&-\sign(\sigma)\sum_{i=1}^n(-1)^i[[[y_1,\dots,y_n],x_{\sigma(i)},z_1,\dots, z_{n-2}],x_{\sigma(1)},\dots, \widehat{x_{\sigma(i)}}, \dots, x_{\sigma(n)}].
	\end{align*}
	By Lemma \ref{finalrel} this is equal to $0$.
	
		In order to show that relation (\ref{55}) of  $W$ is mapped under $\Theta$ to a relation of $\overline{\text{Lie}}_n(m)$, it suffices by the definition given in (\ref{ThetaG2}) to show the following identity in $\overline{\text{Lie}}_n(m)$
	\begin{align}\label{ThetaG21} &\sum_{i=1}^{n}(-1)^{i}[[[y_1,\dots,y_n],x_i , z_1,\dots, z_{n-2}],x_1,\dots,\widehat{x_i},\dots, x_n] \\\nonumber=&\sign(\sigma)\sum_{i=1}^{n}(-1)^{i}[[[y_{\sigma(1)},\dots,y_{\sigma(n)}],x_i , z_1,\dots, z_{n-2}],x_1,\dots,\widehat{x_i},\dots, x_n] \end{align}
	for any permutation $\sigma \in \mathfrak{S}_n$. From relation (\ref{51}) of $W(1)$ applied to   \[[[[y_1,\dots,y_n],x_i , z_1,\dots, z_{n-2}],x_1,\dots,\widehat{x_i},\dots, x_n] \]
	we have
	\begin{align*}&[[[y_1,\dots,y_n],x_i , z_1,\dots, z_{n-2}],x_1,\dots,\widehat{x_i},\dots, x_n] \\\nonumber=& \sign(\sigma) [[[y_{\sigma(1)},\dots,y_{\sigma(n)}],x_i , z_1,\dots, z_{n-2}],x_1,\dots,\widehat{x_i},\dots, x_n]. \end{align*}
	Taking the alternating sum with respect to $i=1,\dots, n$ we obtain eq. (\ref{ThetaG21}).
	
		Similarly to the previous case, in order to show that relation (\ref{56}) of  $W$ is mapped under $\Theta$ to a relation of $\overline{\text{Lie}}_n(m)$, it suffices to show the following identity in $\overline{\text{Lie}}_n(m)$
	\begin{align}\label{ThetaG22} &\sum_{i=1}^{n}(-1)^{i}[[[y_1,\dots,y_n],x_i , z_1,\dots, z_{n-2}],x_1,\dots,\widehat{x_i},\dots, x_n] \\\nonumber=& \sign(\tau)\sum_{i=1}^{n}(-1)^{i}[[[y_{1},\dots,y_{n}],x_i , z_{\tau(1)},\dots, z_{\tau(n-2)}],x_1,\dots,\widehat{x_i},\dots, x_n] \end{align}
	for any permutation $\tau \in \mathfrak{S}_{n-2}$. From relation (\ref{52}) of $W(1)$ applied to   \[[[[y_1,\dots,y_n],x_i , z_1,\dots, z_{n-2}],x_1,\dots,\widehat{x_i},\dots, x_n] \]
	we have
	\begin{align*}&[[[y_1,\dots,y_n],x_i , z_1,\dots, z_{n-2}],x_1,\dots,\widehat{x_i},\dots, x_n] \\\nonumber=&\sign(\tau)[[[y_{1},\dots,y_{n}],x_i , z_{\tau(1)},\dots, z_{\tau(n-2)}],x_1,\dots,\widehat{x_i},\dots, x_n], \end{align*}
	where the element $x_i$ remains fixed. Hence (\ref{ThetaG22}) follows.
	
	From the definition of $\Theta$ (see (\ref{ThetaG2})), it follows that the image of the relation (\ref{57}) is equal to the relation (R5) of $\overline{\text{Lie}}_n(m)$. 
	
	Thus far we have shown that $\Theta:W \to \overline{\text{Lie}}_n(m)$ is a well defined map. Since $\overline{\text{Lie}}_n(m)$ is a quotient of $W(1)$ and the (G1) generate $W(1)$, it is clear from the definition of $\Theta$ that $\Theta$ is surjective.
	
	Finally, from the definition of $\Theta$ it is clear that the image under $\Theta$
	\begin{itemize}
		\item of (R1) of Lemma \ref{genrel1} is equal to (R1) in $\overline{\textup{Lie}}_n(m)$,
		\item of (R2) of Lemma \ref{genrel1} is equal to (R4) in $\overline{\textup{Lie}}_n(m)$, and
		\item of (R3) of Lemma \ref{genrel1} is equal to zero.
	\end{itemize}
From Lemma \ref{genrel1} it follows that $\Theta$ induces a surjective linear map $\text{Lie}_n(m) \to \overline{\text{Lie}}_n(m)$. The proof is complete.

 	\end{proof}
\subsection{A presentation associated to the map $\gamma_1+\gamma_2 +\gamma_3$}
For the remainder of this section, let $\lambda=(n,n-1,n-1)$, $\mu=(n+1,n-1,n-2)$ and $\nu=(n,n,n-2)$.

Applying the functor $\Omega$ that we discussed in Section \ref{Omega} to the map $\gamma_1 +\gamma_2 + \gamma_3$ of Theorem \ref{main2}, we obtain a map of $G$-modules
\begin{equation}\label{omegaonmaps}
	\Lambda(\lambda) \oplus \Lambda(\lambda) \oplus  \Lambda(\nu) \xrightarrow{\Omega(\gamma_1) + \Omega(\gamma_2)+ \Omega(\gamma_3)} \Lambda(\lambda).
\end{equation}
From Definition \ref{mapg}, Definition \ref{mapg3}, the definition of the maps $\phi_S$ and $\psi_S$ in (\ref{phis}) and (\ref{psiS}), and the description of the functor $\Omega$ in Section \ref{Omega} (especially item 3 that involves sign changes), it is straightforward to verify that the maps $\Omega(\gamma_1)$, $\Omega(\gamma_2)$ and  $\Omega(\gamma_3)$ are as follows,
\begin{align*}
	\Omega(\gamma_1)&=1_{\Lambda(\lambda)}-\psi_{S(2)},\\
	\Omega(\gamma_2)&=1_{\Lambda(\lambda)}-\psi_{S(3)}+\psi_{S(4)},\\
	\Omega(\gamma_3)&=\psi_{Q(1)} + \psi_{Q(2)}.
\end{align*}

We will need to compute these maps on basis elements of $\Lambda(\lambda)$ (in the cases of $\Omega(\gamma_1)$ and $\Omega(\gamma_2)$) and on basis elements of $\Lambda(\nu)$ (in the case of $\Omega(\gamma_3)$). 

Recall that we have the basis $ \{ e_1, \dots, e_N \} $ of the natural $G$-module $V$. So let \[w=x \otimes y \otimes z \in \Lambda(\lambda),\] where $x=x_1 \cdots x_n \in \Lambda^{n}$, $y=y_1 \cdots y_{n-1} \in \Lambda^{n-1}$, $z \in \Lambda^{n-1}$ and $x_i, y_j \in \{e_1, \dots, e_N\}$. For $i \in \{1, \dots, n\}$ we let \[x[i] = x_1 \cdots \widehat{x_i} \cdots x_n \in \Lambda^{n-1},\] where  $\widehat{x_i}$ means that $x_i$ is omitted. For $j \in \{1, \dots, n-1\}$ we define $y[j] \in \Lambda ^{n-2}$ in a similar manner. It is easy to check using the definition of $\psi_S$ in Definition \ref{defpsiS} and the definition of the tableaux $S(i)$ in Definition \ref{mapg}, that 
\begin{align}\label{omegagamma}
	\Omega(\gamma_1)(w)&=w-\sum_{i=1}^n(-1)^{i-1}x_iy \otimes x[i] \otimes z,\\\label{omegaphi}
	\Omega(\gamma_2)(w)&=w-x\otimes z \otimes y+\sum_{i=1}^{n-1}(-1)^{i-1}\sum_{j=1}^n(-1)^{j-1}y_iz\otimes x_jy[i] \otimes x[j].
\end{align}
Likewise, for \[u=x' \otimes y' \otimes z' \in \Lambda(\nu),\] where $x'=x_1 \cdots x_n \in \Lambda ^{n}$, $y'= y_1 \cdots y_n \in \Lambda^ {n}$, $z' \in \Lambda^{n-2}$  and $x_i, y_i \in \{e_1, \dots, e_N\}$, we have \begin{equation}\label{omegagamma3}
\Omega(\gamma_3)(u)=\sum_{i=1}^{n} x' \otimes y_i z' \otimes y'[i] + \sum_{i=1}^{n} y' \otimes x_i z' \otimes x'[i].
\end{equation}

Next we want to apply the Schur functor. Suppose $N \ge m$. Recall from \cite{Gr}, the Schur functor is a functor $f$ from the category of homogeneous polynomial representations of  $G$ of degree $m$ to the category of left $\mathfrak{S}_m$-modules. For $M$ an object in the first category,  $f(M)$ is the weight subspace $M_\alpha$ of $M$, where $\alpha = (1^m, 0^{N-m})$ and for $\theta :M \to Q$ a morphism in the first category, $f(\theta)$  is the restriction $M_\alpha \to Q_\alpha$ of $\theta$. Let us denote the conjugate of a partition $\xi$ by $\xi'$. The Specht module corresponding to a partition $\xi$ of $m$ will be denoted by $S^{\xi}$ and the space of column tabloids corresponding to $\xi$ will be denoted by ${\tilde{M}}^\xi$. See \cite[Chapter 7.4]{F} for the later. It is well known that $f$ is an exact functor such that $f(L_{\xi})=S^{\xi'}$ and $f(\Lambda^{\xi}) = \tilde{M}^{\xi'}$ for any partition $\xi$ of $m$, where $L_{\xi}$ is the Schur module introduced in Section 2.2.

Applying the Schur functor to the map $\Lambda(\lambda) \oplus \Lambda(\lambda) \oplus  \Lambda(\nu) \xrightarrow{\Omega(\gamma_1) + \Omega(\gamma_2)+ \Omega(\gamma_3)} \Lambda(\lambda)$ in (\ref{omegaonmaps}) we obtain a map of $\mathfrak{S}_m$-modules \begin{equation}\label{cokomega} \tilde{M}^{\lambda'} \oplus \tilde{M}^{\lambda'}  \oplus  \tilde{M}^{\nu'}\xrightarrow{f(\Omega(\gamma_1)) + f(\Omega(\gamma_2)) + f(\Omega(\gamma_3))} \tilde{M}^{\lambda'}.\end{equation} If an element $x\otimes y \otimes z$ of $\Lambda(\lambda)$ is in $\tilde{M}^{\lambda'}$, we denote by $x|y|z$ its image in the cokernel of (\ref{cokomega}).
\begin{lem}\label{pres1} The cokernel of the map (\ref{cokomega}) has a presentation with generators $x|y|z$ and relations 
	\begin{align}\label{45}&x|y|z=\sum_{i=1}^n(-1)^{i-1}x_iy | x[i] | z,\\\label{46}
		&x|y|z=x | z | y-\sum_{i=1}^{n-1}(-1)^{i-1}\sum_{j=1}^n(-1)^{j-1}y_iz | x_jy[i] | x[j],\\\label{47}&\sum_{i=1}^{n} x' | y_i z' | y'[i] + \sum_{i=1}^{n} y' | x_i z' | x'[i] =0,\end{align}
	where \begin{itemize} \item $x \in \Lambda^n$ and $y,z \in \Lambda^{n-1}$ run over all elements of the form $x=x_{1} \cdots x_{n}$, $y=y_{1} \cdots y_{{n-1}}$, $z=z_{1} \cdots z_{n-1}$ such that $x_1, \dots x_{n}, y_1, \dots, y_{n-1}, z_1, \dots, z_{n-1}$ is permutation of $e_1,e_2, \dots, e_{3n-2},$ and
		\item $x', y' \in \Lambda^n$ and $z' \in \Lambda^{n-2}$ run over all elements of the form $x'=x_{1} \cdots x_{n}$, $y'=y_{1} \cdots y_{{n}}$, $z'=z_{1} \cdots z_{n-2}$ such that $x_1, \dots x_{n}, y_1, \dots, y_{n}, z_1, \dots, z_{n-2}$ is permutation of $e_1,e_2, \dots, e_{3n-2}.$
		\end{itemize}
\end{lem}
\begin{proof}
	This follows from the previous discussion and the equalities (\ref{omegagamma}), (\ref{omegaphi}), (\ref{omegagamma3}).
\end{proof}
\subsection{Proof of Theorem \ref{mainspecht}} \begin{proof}Recall that we are assuming $N \ge m=3n-2$ and $n \ge 2$. Also we have the partitions $\lambda=(n,n-1,n-1)$, $\mu=(n+1,n-1,n-2)$ and $\nu=(n,n,n-2)$.

Consider the map of $\mathfrak{S}_m$-modules
\begin{align*}h:\tilde{M}^{\lambda'} &\to \overline{\text{Lie}}_n(m), \\ x_1 \cdots x_n|y_1 \cdots y_{n-1}|z_1 \cdots z_{n-1} &\mapsto [[[x_1,\dots,x_n],y_1,\dots,y_{n-1}],z_1, \dots, z_{n-1}]. \end{align*}
From Lemma \ref{pres1} it follows that $h$ induces a map of $\mathfrak{S}_m$-modules \begin{equation}\label{lieiso} \coker(\tilde{M}^{\lambda'} \oplus \tilde{M}^{\lambda'} \oplus \tilde{M}^{\nu'} \xrightarrow{f(\Omega(\gamma_1)) + f(\Omega(\gamma_2))+ f(\Omega(\gamma_3))} \tilde{M}^{\lambda'}) \to \overline{\text{Lie}}_n(m) \end{equation}
which is an isomorphism, since $h$ carries (\ref{45}), (\ref{46}) and (\ref{47}) to (R1), (R4) and (R5) respectively. 
Hence, from Theorem \ref{main2} it follows that $\overline{\text{Lie}}_n(m) \simeq S^{\lambda'} \oplus S^{\mu'}$ as $\mathfrak{S}_m$-modules. From Lemma \ref{presentationlanke} we have $\text{Lie}_n(m) \simeq S^{\lambda'} \oplus S^{\mu'}$ as $\mathfrak{S}_m$-modules as desired.\end{proof}

We have seen that the map in (\ref{lieiso}) is an isomorphism of $\mathfrak{S}_m$-modules.  From this and Remark \ref{rem46} it follows that we  have another isomorphism of $\mathfrak{S}_m$-modules, 
	\begin{equation}\label{lieiso2} \coker(\tilde{M}^{\lambda'} \oplus \tilde{M}^{\lambda'} \xrightarrow{f(\Omega(\gamma_1)) + f(\Omega(\gamma_2))} \tilde{M}^{\lambda'}) \to \text{Lie}_n(m). \end{equation}
	Thus we obtain the following corollary.

\begin{cor} For $m=3n-2$ and $n \ge 2$ we have $\textup{Lie}_{n}(m) \simeq W(1)/\spn(R1, R4)$ as $\mathfrak{S_m}$-modules.
	\end{cor}

\section*{Acknowledgments}
We are  very grateful to the anonymous referees for many detailed and constructive  comments and suggestions that helped to greatly improve the paper and, in particular, for pointing out a gap in the proof of a previous form of Lemma 5.10 in an earlier version of the paper. 

The research project is implemented in the framework of H.F.R.I. Call “Basic research Financing Horizontal support of all Sciences)” under the National Recovery and Resilience Plan “Greece 2.0” funded by the European Union Next Generation EU, H.F.R.I.  
Project Number: 14907.

% ------------------------------------------------------------------------
\end{document}